\documentclass[a4paper, 10pt]{amsart}
\usepackage{amsmath}
\usepackage{amsthm}
\usepackage{amssymb}
\usepackage{framed}
\usepackage[hidelinks]{hyperref}
\usepackage{xcolor}
\usepackage{graphicx}
\usepackage[a4paper,  margin=2.5cm]{geometry}

\newtheorem{de}{Definition}
\newtheorem{ex}{Example}

\newtheorem{pr}{Proposition}
\newtheorem{lm}{Lemma}
\newtheorem{co}{Corollary}
\newtheorem{te}{Theorem}
\newtheorem{re}{Remark}
\theoremstyle{definition}

\newcommand{\dif}{\mathrm{d}}
\newcommand{\ii}{\mathrm{i}}

\newcommand{\abs}[1]{\vert #1 \vert}
\newcommand{\ov}{\overline}

\newcommand{\CC}{\mathbb{C}}

\newcommand{\RR}{\mathbb{R}}
\newcommand{\ZZ}{\mathbb{Z}}
\newcommand{\NN}{\mathbb{N}}
\newcommand{\Ss}{\mathbb{S}}
\newcommand{\Tt}{\mathbb{T}}

\DeclareMathOperator{\di}{div}
\DeclareMathOperator{\cu}{curl}

\DeclareMathOperator{\Ric}{Ric}

\begin{document}
\title[Contact structures and Beltrami fields]{Contact structures and Beltrami fields on the torus\\ and the sphere}

\author{D. Peralta-Salas}
\address{Instituto de Ciencias Matem\'aticas, Consejo Superior de
 Investigaciones Cient\'\i ficas, 28049 Madrid, Spain}
\email{dperalta@icmat.es}

\author{R. Slobodeanu}
\address{Faculty of Physics, University of Bucharest, P.O. Box Mg-11, RO--077125 Bucharest-M\u agurele, Romania}
\email{radualexandru.slobodeanu@g.unibuc.ro}




\begin{abstract}
\noindent We show the existence of overtwisted contact structures on the (round)
3-sphere and the (flat) 3-torus for which the ambient metric is weakly compatible. Our proofs are based on the construction of $\mathbb S^1$-invariant nonvanishing curl eigenfields on $\mathbb S^3$ and $\mathbb T^3$ using suitable families of Jacobi polynomials and eigenfunctions of the Laplacian with contractible nodal sets, respectively. As a consequence, we show that the contact sphere theorem of Etnyre, Komendarczyk and Massot (2012) does not hold for weakly compatible metrics as it was suggested. Using ideas from the theory of harmonic maps, we also provide an alternative proof of a theorem of Gluck and Gu (2001) on the geometric rigidity of tight contact structures on the $3$-sphere, i.e., any contact form in $\mathbb S^3$ admitting a compatible metric that is the round one is isometric, up to a constant factor, to the standard (tight) contact form.
\end{abstract}

\maketitle

\section{Introduction}\label{S:intro}
Since 1889 when Eugenio Beltrami highlighted them in his paper ``Considerations on Hydrodynamics'', the divergence-free vector fields $V$ aligned with their own curl, i.e.
$$
\cu V=fV
$$
for some function $f$, have found far reaching applications in magnetohydrodynamics~\cite{Marsh,reed}, and in the mathematical analysis of the Euler and Navier-Stokes equations~\cite{Const}. Known also as force-free magnetic fields in solar physics, or helical flows in fluid mechanics, Beltrami fields are used to describe equilibria in magnetic relaxation processes~\cite{Moff}, as building blocks to construct time-dependent solutions exhibiting dissipation~\cite{delellis} and vortex reconnections~\cite{eps}, or as special steady fluid flows with complex dynamics~\cite{Annals, Acta}, able to provide us with insights about Lagrangian turbulence. We refer to~\cite{arn,ps}, for a broader view on (topological) hydrodynamics and the role of Beltrami fields.

On a Riemannian 3-manifold $(M,g)$, the 1-form $\alpha$ that is dual (using the metric) to a \emph{Beltrami field} $V$ satisfies:
\begin{equation}\label{defbelt}
\ast \dif  \alpha = f \alpha\,, \qquad \delta \alpha =0\,,
\end{equation}
for some (say $C^\infty$) function $f: M \to \RR$, where $\ast$ is the Hodge star operator and $\delta = \ast \dif \ast$ is the codifferential. From the first equation it follows that
$$\alpha \wedge \dif \alpha=f\abs{\alpha}^2 \mathrm{vol}_g\,,$$
so, in particular, if $\alpha$ and $f$ are nowhere vanishing, then $\alpha \wedge \dif \alpha \neq 0$, condition well-known to geometers as defining a \emph{contact form} on $M$. In this case we say that the contact structure defined by the kernel of $\alpha$ is associated with the Beltrami field $V$. Conversely, as proven in~\cite{etn}, for any contact form  its associated Reeb field is a Beltrami field (with constant proportionality factor) with respect to some Riemannian metric on $M$. This correspondence gives rise to a fruitful interplay between contact geometry and hydrodynamics, with important benefits for both fields~\cite{EG00,CMPP}.

A major goal in contact topology is the understanding of overtwisted and tight contact structures, the former being defined in terms of the existence of an embedded disk which, along its boundary, is tangent to the contact planes. On closed oriented 3-manifolds the overtwisted contact structures are classified (up to contact isotopy) by the homotopy classes of their 2-plane fields~\cite{Eli}. In view of the contact-Beltrami correspondence mentioned above, it is then natural to ask which type of contact structures (tight or overtwisted) are associated with nonvanishing Beltrami fields on a given Riemannian 3-manifold (Question $7.4$ in \cite{kom}), and which homotopy classes (of 2-plane fields associated to contact structures) can be represented by them. This turns out to be related to the compatibility between the metric and the contact structure, defined as (see~\cite{ekm}):

\begin{de}\label{defcomp}
Let $(M, \alpha)$ be a contact $3$-manifold. A Riemannian metric $g$ on $M$ is called \emph{weakly compatible} with the contact structure if there exists a positive function $f$ on $M$ such that $\ast \dif \alpha  = f \alpha$ (where $\ast$ is computed using the metric $g$). If, in addition, $\abs{\alpha}=1$ and $f$ is a positive constant, then $g$ is called \emph{compatible} with the contact structure.
\end{de}

We recall that the notion of compatible metric was first introduced in~\cite{ch} by requiring that $\ast \dif \alpha  = 2 \alpha$ and $\abs{\alpha}=1$ (so $f\equiv 2$ above), and in this case we recover also the definition of \emph{associated metric}, customarily used in contact geometry, see e.g.~\cite{blair}.

We see from Definition~\ref{defcomp} that if $V$ is a nonvanishing curl eigenfield on a Riemannian manifold $(M,g)$, i.e.,
$$
\cu V=\lambda V
$$
for some constant $\lambda \neq 0$, (so a Beltrami field with constant proportionality factor), the metric $g$ is weakly compatible with the contact structure defined by the dual $1$-form $\alpha$. If moreover $\abs{V}=1$, then $g$ is compatible.

The curl eigenfields define, on any closed $3$-manifold, an $L^2$-basis for the exact divergence-free vector fields (in particular, the spectrum of $\cu$ is discrete). However, it is widely open if there exist nonvanishing curl eigenfields on any $(M,g)$, or even assuming their existence, it is not known with which contact structures (tight or overtwisted) and homotopy classes they are compatible. (The only general result~\cite{TAMS} is that for a generic metric, the zero set of the curl eigenfields consists of hyperbolic points.) Partial results have been proved recently concerning geometric restrictions for the existence of compatible metrics: the \emph{Contact Sphere Theorem} proved in~\cite{ekm} and, with improved pinching constant, in \cite{ge}.
\begin{te}[Etnyre, Komendarczyk and Massot, 2012]\label{T:pinch}
Let $(M,\alpha)$ be a closed contact 3-manifold. If there exists a compatible metric $g$ with pinched sectional curvature $$0<\frac{4}{9}k_{max}<\mathrm{sec}(g)\leq k_{max}\,,$$
then the contact structure is tight and the manifold is covered by a $3$-sphere.
\end{te}

This remarkable result implies, in particular, that any curl eigenfield with constant norm on a Riemannian manifold with pinched sectional curvature defines a tight contact structure. It is then natural to ask if the same result holds for any nonvanishing curl eigenfield. In this direction,  in~\cite[Remark 1.6]{ekm} it was suggested that Theorem~\ref{T:pinch} might also hold for weakly compatible metrics: ``it is reasonable to hope that such a theorem (the contact sphere theorem) holds for weakly compatible metrics, but possibly with weaker bounds''.

Our goal in this paper is to analyze the type of contact structures (tight or overtwisted, and homotopy classes) associated with nonvanishing curl eigenfields on $\mathbb S^3$ (endowed with the round metric) and $\mathbb T^3$ (endowed with the flat metric). This is tantamount to studying contact structures on $\mathbb S^3$ or $\mathbb T^3$ whose weakly compatible metrics are the canonical ones on these manifolds (actually, in this setting of curl eigenfields, the notion of weak compatibility is slightly stronger than in Definition~\ref{defcomp} because the proportionality factor is constant). In what follows we will always assume that $\mathbb S^3$ and $\mathbb T^3$ are endowed with their canonical metrics.

Our first main result concerns curl eigenfields on $\mathbb S^3$. We recall that the spectrum of $\cu$ is given by $\lambda=\pm(k+2)$ where $k$ is a nonnegative integer. It is well known that the Reeb field of the standard (resp. anti-standard) contact form on $\mathbb S^3$ is a curl eigenfield with $\lambda=2$ (resp. $\lambda=-2$). These fields are known as the Hopf and anti-Hopf fields, respectively (for details see Section~\ref{S:sphere}). In fact, any curl eigenfield with the lowest eigenvalue $\lambda=2$ (resp. $\lambda=-2$) is isometric (possibly up to a constant proportionality factor) to the Hopf (resp. anti-Hopf) field. The following theorem provides a partial answer to the problem of characterizing those contact structures (including their homotopy classes) associated with curl eigenfields on $\mathbb S^3$; in the statement we use the concept of \emph{Hopf class}, which is introduced in Section~\ref{S:sphere}.

\begin{te} \label{maint}
Any nonvanishing curl eigenfield on $\mathbb S^3$ has even eigenvalue $\lambda=2m$, $m\in\{\pm 1,\pm 2, \cdots\}$. Moreover, for each $|m|\geq 2$ there exists a nonvanishing curl eigenfield $V_m$ whose associated contact structure is overtwisted. The homotopy classes of the corresponding contact plane fields have Hopf invariant
$$
\text{Hopf invariant}=\frac{\text{sign}(m)(-1)^{m+1}-1}{2}\,.
$$
\end{te}

The proof of this result is based on the construction of explicit $\Ss^1$-invariant nonvanishing curl eigenfields on $\mathbb S^3$ using an ansatz introduced in~\cite{kkpspaper} and suitable families of Jacobi polynomials. The characterization of the underlying contact structures is obtained using a criterion of Giroux~\cite{gir} for tightness and explicit constructions of contact homotopies with model fields. This theorem shows, in particular, that there exist overtwisted contact structures associated with nonvanishing Beltrami fields on $\mathbb S^3$, thus establishing that Theorem~\ref{T:pinch} does not hold for weakly compatible metrics:

\begin{co}
For each Hopf class $0$ or $-1$, there exists an overtwisted contact structure on $\mathbb S^3$ admitting a weakly adapted metric that is the round one.
\end{co}

The topological rigidity of tight structures on $\mathbb S^3$ (any tight contact structure is contact isotopic to the standard one) has a nice \emph{geometric rigidity} counterpart, i.e., any constant norm Beltrami field on $\Ss^3$ is isometric to the Hopf (or anti-Hopf) field~\cite[Theorem A]{glu}. We revisit this result by providing a completely new proof of the following statement:

\begin{pr}\label{maint2}
Let $V$ be a curl eigenfield on $\mathbb S^3$ with $\abs{V}=1$. Then $\lambda=\pm 2$ and $V$ is isometric to the Hopf (or anti-Hopf) field.
\end{pr}

While in the proof of~\cite{glu} one associates a distance-decreasing holomorphic map from $\Ss^2$ to $\Ss^2$ to the Beltrami field $V$, and then applies a Liouville-type theorem to show that there are no nontrivial maps of this type, in our proof we associate to $V$ a map from $\Ss^3$ to $\Ss^2$ that turns out to be a submersive polynomial harmonic morphism and therefore isometrically related to the Hopf map by a standard classification result~\cite{eells}.

Concerning the torus $\mathbb T^3$, we show that there exists a non-vanishing curl eigenfield with associated tight contact structure for each eigenvalue, and that for infinitely many eigenvalues there are curl eigenfields that are associated with overtwisted contact structures. The proof of the first part is based on a simple explicit construction, while the second part is more involved and is based on the construction of $\mathbb S^1$-invariant curl eigenfields using eigenfunctions of the Laplacian on $\mathbb T^2$ with contractible nodal sets. In the statement, the terminology \emph{homotopically trivial field} means that the field is homotopic (via nonvanishing fields) to the coordinate field $\partial_{x_1}$.

\begin{te}\label{maint3}
For each eigenvalue $\lambda$ of the curl operator on $\mathbb T^3$, there exists a nonvanishing curl eigenfield $V_\lambda$ which is homotopically trivial and whose associated contact structure is tight. Moreover, all tight contactomorphic classes are realized this way. Furthermore, there exist infinitely many eigenvalues $\{\lambda_\ell\}_{\ell \in \NN^*}$ and corresponding eigenfields $V_\ell$ such that, for each $\ell$, the contact structure associated with $V_\ell$ is overtwisted.
\end{te}

The organization of this paper is as follows. Section~\ref{S:prelim} is divided in three subsections; Subsection~\ref{SS:contact} recalls some basic notions of contact geometry and in Subsection~\ref{SS:Belta} we prove Proposition~\ref{gencont}, which provides a criterion to show when two contact structures associated with curl eigenfields are contactomorphic; in Subsection~\ref{nodal_sect} we prove two results that are instrumental to prove Theorem~\ref{maint3}, i.e. Propositions~\ref{density} and~\ref{nodalcirc}, which establish some useful properties of the nodal sets of the eigenfunctions of the Laplacian on $\mathbb T^2$ and $\mathbb S^2$. Theorem~\ref{maint} and Proposition~\ref{maint2} on $\mathbb S^3$ are proved in Sections~\ref{S:sphere} and~\ref{S:const}, respectively, while the proof of Theorem~\ref{maint3} is presented in Section~\ref{S:torus}. Some examples and final considerations are provided in Section~\ref{S:final}, including an application to the construction of planar open book decompositions.

\section{Notation and preliminary results}\label{S:prelim}

The goal of this section is to introduce first (in Subsection~\ref{SS:contact}) some notions of contact geometry that will be used without further mention in the next sections. In Subsection~\ref{SS:Belta} we recall the classical contact-Beltrami correspondence and present a criterion to analyze when two contact structures associated with different curl eigenfields are contactomorphic (cf. Proposition~\ref{gencont}). Finally, in Subsection~\ref{nodal_sect} we prove two results on the eigenfunctions of the Laplacian in $\mathbb S^2$ and $\mathbb T^2$, of interest in themselves, that will be key in Section~\ref{S:torus} (cf. Propositions~\ref{density} and~\ref{nodalcirc}).

\subsection{Basics of contact geometry}\label{SS:contact}
Let $M$ be a closed (oriented, smooth) $3$-manifold. A contact form on $M$ is a 1-form $\alpha$ such that $\alpha \wedge \dif \alpha \neq 0$. It is customary to assume that $\alpha \wedge \dif \alpha > 0$, i.e., $\alpha$ is a positive contact form and $\alpha\wedge \dif \alpha$ defines a volume form on $M$ (otherwise, the contact form is said to be negative). A $2$-plane field $\zeta \subset TM$ is a (coorientable) contact structure on $M$ if there is a contact form $\alpha$ such that $\zeta=\ker \alpha$. Such a $2$-plane field is a maximally non-integrable distribution on $M$, also called a contact distribution. The Reeb vector field $R$ associated to a contact form is (uniquely) determined by the conditions $\alpha(R) =1$, $i_R\dif\alpha=0$.

A contact structure is called \textit{overtwisted} if there exists an embedded disk which is tangent to the contact distribution along its boundary; if such an embedded disk does not exist, the contact structure is called \textit{tight}.

The following equivalence relations between two contact structures $\zeta_0$ and $\zeta_1$ (with defining contact forms $\alpha_0$ and $\alpha_1$, respectively) are usually introduced in the literature:

\begin{itemize}
\item \emph{Contactomorphic}: if there is a diffeomorphism $\varphi$ such that $\dif \varphi (\zeta_0)=\zeta_1$. In terms of the associated 1-forms: $\varphi^* \alpha_1 = f \alpha_0$, for some nowhere zero function $f$.

\item \emph{Contact homotopic}: if there is a smooth family of contact structures $\zeta_t$, $t\in[0,1]$ connecting $\zeta_0$ with $\zeta_1$. In terms of the associated 1-forms: there exists a $1$-parameter family $\alpha_t$ such that $\alpha_t \wedge \dif \alpha_t \neq 0$, for all $t\in[0,1]$.

\item \emph{Contact isotopic}: if there is a smooth family of diffeomorphisms $\varphi_t$, $t\in[0,1]$ with $\varphi_0=id$ and $\dif\varphi_1(\zeta_0)=\zeta_1$. In terms of the associated 1-forms:  $\varphi^*_1 \alpha_1 = f \alpha_0$, for some positive function $f$.

\item \emph{Homotopic through plane fields}: if there is a smooth family of plane fields $\zeta_t$, $t\in[0,1]$ connecting $\zeta_0$ with $\zeta_1$. In terms of the associated 1-forms:  there exists a $1$-parameter family of nondegenerate $1$-forms $\alpha_t$ (i.e. $\alpha_t \neq 0$ for all $t\in[0,1]$). The homotopy through plane fields is equivalent (after fixing a global parallelism of $M$) to the homotopy of the associated maps $\varphi_{R} : M \to \Ss^2$ defined as $\varphi_R (p):=\frac{1}{\abs{R_p}}R_p$ for any $p\in M$. Recall that $R$ is the Reeb vector field associated with the contact form. The map $\varphi_R$ is called the \emph{Gauss map} of the contact distribution. We recall that the Euler class $e(\zeta)\in H^2(M;\ZZ)$ of a contact distribution $\zeta$ can be computed using the Gauss map $\varphi_R$ as $e(\zeta)=\varphi_R^*(e(T\Ss^2))$.
\end{itemize}

Besides the obvious implications ``contact isotopic $\Rightarrow$ contactomorphic'' and ``contact homotopic $\Rightarrow$ homotopic through plane fields'', let us recall two classical results that connect in a non obvious way the equivalence relations defined above:

\begin{te}[Gray stability]
Contact homotopic $\Rightarrow$ contact isotopic.
\end{te}
\begin{te}[Eliashberg classification~\cite{Eli}]
For overtwisted contact structures, homotopic through plane fields $\Leftrightarrow$ contact homotopic. In particular, the contact isotopy classes of overtwisted contact structures are indexed by the Hopf invariant of (the map $\varphi_R : M \to \Ss^2$ associated to) the contact structure, which is an integer if $M=\Ss^3$.
\end{te}

\subsection{Beltrami fields and contact structures}\label{SS:Belta}

The curl of a vector field $V$ on a Riemannian $3$-manifold $(M,g)$ is defined as the unique vector field $\cu V$ that satisfies
$$
i_{\cu V}\mathrm{vol}_g=\dif\alpha\,,
$$
where $\mathrm{vol}_g$ is the Riemannian volume form, and $\alpha$ is the dual $1$-form of $V$ (computed with the metric). The curl defines an essentially self-adjoint operator on the space of $L^2$ exact divergence-free vector fields on $M$ (we recall that a field is exact if it is the curl of another field), and its spectrum is discrete. In particular, the eigenfields of curl,
$$
\cu V=\lambda V
$$
for some real $\lambda\neq 0$, form a basis of the aforementioned space. In terms of the dual $1$-form, the eigenfield equation reads as
$$
\ast \dif  \alpha = \lambda \alpha\,.
$$

Curl eigenfields are Beltrami fields with constant proportionality factor (i.e., $f=\lambda$ in Equation~\eqref{defbelt}). Most of the literature on Beltrami fields concerns curl eigenfields because they are more flexible and their global existence is much easier to establish (see~\cite{ARMA} for an account on Beltrami fields with nonconstant proportionality factor).

As noticed in Section~\ref{S:intro}, for the dual $1$-form of a Beltrami field we have $\alpha \wedge \dif \alpha=f\abs{\alpha}^2 \mathrm{vol}_g$ and therefore, if the field and the factor $f$ are nonvanishing, $\alpha$ is a contact form on $M$ and
$$R:=\frac{V}{\abs{V}^2}$$
is the associated Reeb field. Conversely, given a contact form $\alpha$ on $M$ with the associated Reeb vector field $R$, there exists~\cite{blair} an endomorphism $\phi$ of $TM$ and a Riemannian metric $g$ such that, for any vector fields $X,Y$ on $M$,
\begin{equation}\label{assocmet}
\phi^2= - I + \alpha \otimes R\,, \quad \alpha(Y)=g(R, Y)\,, \quad \tfrac{1}{2}\dif \alpha (X, Y) = g(X, \phi Y)\,.
\end{equation}
For such a metric $g$, one can prove that $\tfrac{1}{2}\alpha \wedge \dif \alpha= \mathrm{vol}_g$ and therefore $\ast \dif \alpha = 2 \alpha$. Accordingly, the Reeb field of the given contact form is a (nonvanishing) Beltrami field (in fact, a curl eigenfield). This correspondence between Reeb and Beltrami fields was developed by Etnyre and Ghrist~\cite{etn}, and can be summarized as follows:
\begin{pr}
Let $M$ be a Riemannian 3-manifold. Any nonvanishing
Beltrami field on $M$ with positive proportionality factor is, up to rescaling, the Reeb field for some contact form on $M$. Conversely, given a
contact form $\alpha$ on $M$ with associated Reeb field $R$, any nonzero rescaling of $R$ is a (nonvanishing) curl eigenfield for some Riemannian metric on $M$ (which is weakly compatible with the contact structure).
\end{pr}

The main result of this subsection provides a sufficient condition under which two contact structures associated with curl eigenfields (with the same eigenvalue) are contact homotopic (so contact isotopic and with homotopic Gauss maps). This result is instrumental to find out whether a given Beltrami field is tight or overtwisted, by comparing it with the models that we will introduce in Sections~\ref{S:sphere} and~\ref{S:torus}. Specifically, Proposition~\ref{gencont} is used in Examples~\ref{nonKKPS2},~\ref{torus_ex} and~\ref{ex_final} presented in Section~\ref{S:final} to characterize the contact structures associated with different Beltrami fields. In its statement we denote by $C_+$ and $C_{-}$ the sets of points where the curl eigenfields $V$ and $W$ are positively, resp. negatively, aligned.

\begin{re}
We observe that the sets $C_+$ and $C_-$ were already used in~\cite{duf} to measure the ``distance'' in homotopy classes between two vector fields $V$ and $W$. In particular,~\cite[Lemma 19]{duf} shows that if either $C_+$ or $C_-$ is empty, then $V$ and $W$ are homotopic (through nonvanishing fields). However, we do not use this result in our proof, which consists in finding an explicit contact homotopy between the associated contact forms (notice that the homotopy in~\cite{duf} is not necessarily of contact type).
\end{re}

\begin{pr}\label{gencont}
On a Riemannian $3$-manifold $(M,g)$, let $V$ and $W$ be nonvanishing curl eigenfields with the same eigenvalue $\lambda\neq 0$. Then: $(i)$ if either $C_+$ or $C_{-}$ is empty, the induced contact structures are contact homotopic; $(ii)$ there exists a positive constant $c_0$ that only depends on $V$ and $W$, such that the contact structure associated with the nonvanishing curl eigenfield $V + cW$ is contact homotopic to the one induced by $V$ provided that $|c|<c_0$.
\end{pr}

\begin{proof}
We first assume that $C_{-}= \emptyset$. Let $\alpha$ and $\beta$ be the (nonvanishing) $1$-forms dual to $V$ and $W$, respectively. They satisfy the equations $\ast \dif\alpha=\lambda \alpha$ and $\ast \dif\beta=\lambda \beta$. Now consider the homotopy $\alpha_t:=t\alpha + (1-t)\beta$, $t\in [0,1]$, given by the linear interpolation between these 1-forms. It is straightforward to compute explicitly the $3$-form $\alpha_t \wedge \dif \alpha_t$:
\begin{equation*}
\begin{split}
\alpha_t \wedge \dif \alpha_t &= (t\alpha + (1-t)\beta)\wedge (t \lambda \ast \alpha + (1-t) \lambda \ast \beta)\\
&= \lambda\big\{ t^2 \abs{\alpha}^2 +  (1-t)^2 \abs{\beta}^2  + 2t(1-t)\langle\alpha, \beta \rangle  \big\} \mathrm{vol}_g\\
&= \lambda \big\{ \left[\abs{\alpha}^2 -2\langle\alpha, \beta \rangle + \abs{\beta}^2 \right] t^2 +
2\left[\langle\alpha, \beta \rangle - \abs{\beta}^2 \right] t
+ \abs{\beta}^2  \big\} \mathrm{vol}_g.
\end{split}
\end{equation*}
Let us suppose that there is a point $x \in M$ at which the above $3$-form vanishes for some $t\in[0,1]$. Then the quadratic polynomial in $t$ that multiplies $\mathrm{vol}_g$ in the above equation (evaluated at $x$), must have non-negative discriminant, which is equivalent to
$$
\langle \alpha, \beta \rangle^2 - \abs{\alpha}^2  \abs{\beta}^2  \geq 0\,.
$$
The Cauchy-Schwarz inequality then implies that $\alpha(x) = \theta \beta(x)$, for some constant $\theta > 0$, due to our hypothesis $C_{-}= \emptyset$. But then, at the point $x$,
$$\alpha_t(x) \wedge \dif \alpha_t(x) = \lambda \big((\theta -1)t+1\big)^2 \abs{\beta}^2 \mathrm{vol}_g\,,$$
and this can be zero only if $\theta \neq 1$ and $t=\frac{1}{1-\theta} \notin [0,1]$, which is a contradiction. We conclude that $\alpha_t$ is a contact form for all $t\in[0,1]$, as we wanted to prove. The case $C_{+}= \emptyset$ is treated analogously by working with $-\beta$ instead of $\beta$.

To prove the second statement, we first assume that $C:=C_+\cup C_-\neq \emptyset$. We introduce the constant
\begin{equation}\label{eq.c0}
0<c_0:=\text{min}_{x\in C}\Big\{\frac{\abs{\alpha}(x)}{\abs{\beta}(x)}\Big\}\,.
\end{equation}
This constant is well defined because the set $C$ is closed (and compact because of the compactness of $M$). It is obvious that the $1$-form $\alpha +c\beta$ is nonvanishing if $|c|<c_0$. Now consider the linear homotopy $\alpha_t:=\alpha + c t\beta$, $t\in [0,1]$. We have:
\begin{equation*}
\begin{split}
\alpha_t \wedge \dif \alpha_t &= (\alpha + ct\beta)\wedge (\lambda \ast \alpha + ct \lambda \ast \beta)= \lambda\big\{ c^2 t^2 \abs{\beta}^2  + 2 c t \langle\alpha, \beta \rangle + \abs{\alpha}^2  \big\} \mathrm{vol}_g.
\end{split}
\end{equation*}
Suppose that this expression vanishes at some point $x\in M$, for some $t\in [0,1]$. Then the quadratic polynomial in $t$ that multiplies $\mathrm{vol}_g$ in the above equation (evaluated at $x$) must have non-negative discriminant, which is again equivalent to $\langle \alpha, \beta \rangle^2 - \abs{\alpha}^2  \abs{\beta}^2  \geq 0$. By Cauchy-Schwarz inequality,
$\alpha(x) = \theta \beta(x)$, for some constant $\theta$. But then, at the point $x$,
$$\alpha_t(x) \wedge \dif \alpha_t(x) = \lambda \big(ct +\theta\big)^2 \abs{\beta}^2 \mathrm{vol}_g\,,$$
and this can be zero only if $\theta =-ct$, so that $\abs{\alpha}(x)/\abs{\beta}(x)=\abs{\theta} = \abs{c} t \leq \abs{c}  < c_0$, which contradicts the definition of the constant $c_0$.

In the case that $C=\emptyset$, it is obvious that the fields $V$ and $V+cW$ (curl eigenfields with the same eigenvalue $\lambda$) are not aligned at any point, so the first part of the proposition implies that the induced contact structures are contact homotopic for any $c\in\mathbb R$.
\end{proof}

\begin{re}\label{complementP3}
As in our setting the orientation on $M$ is fixed (by $\mathrm{vol}_g$) as well as the Reeb field (by the Beltrami field), an induced contact structure is implicitly oriented. So in the case $C_{+} = \emptyset$, where we proved that $\alpha$ is homotopic through contact forms with $-\beta$, from Gray stability we can actually deduce that the contact structure $\zeta_\alpha :=\ker \alpha$ is contact isotopic with $\ov \zeta_\beta$, the contact structure $\zeta_\beta :=\ker \beta$ with opposite orientation. On $M=\Ss^3$, essentially due to vanishing Euler class for any contact distribution, $e(\zeta)=0$, we can nevertheless obtain $\zeta_\alpha$ contact isotopic with $\zeta_\beta$   in the case $C_{+} = \emptyset$ as we can obviously do when $C_{-} = \emptyset$. 
\end{re}

\subsection{Eigenfunctions of the Laplacian on $\mathbb T^2$ and $\mathbb S^2$}\label{nodal_sect}
In this subsection we prove some properties of the nodal sets of the eigenfunctions of the Laplacian on $\mathbb T^2=(\RR/2\pi \ZZ)^2$ and $\mathbb S^2$ (unit sphere in $\mathbb R^3$), both endowed with their canonical metrics. These results, which are of interest in themselves, will be instrumental in Section~\ref{S:torus} to prove Theorem~\ref{maint3}, and in Section~\ref{SS:Raf}.

Eigenfunctions of the Laplacian on $\mathbb T^2$ and $\mathbb S^2$ arise when constructing $\mathbb S^1$-invariant curl eigenfields on $\mathbb T^3$ and $\mathbb S^3$. Roughly speaking, one can obtain nonvanishing solutions to the Beltrami equation by ``lifting'', in a suitable way, generic eigenfunctions of the Laplacian on the quotient space of the $\mathbb S^1$-action. Using this construction and a criterion of Giroux~\cite{gir}, we can study the overtwistedness of the associated contact structures analyzing the nodal sets of the aforementioned eigenfunctions (specifically, their connectedness and contractibility). The connection between dividing sets on certain surfaces (including nodal sets of eigenfunctions) and contact geometry has also been exploited in~\cite{ghri,komtams,komphd,kom,lisi}.

A function $f$ is an eigenfunction of the Laplacian $\Delta:=\di \circ \nabla$ (also called the Laplace-Beltrami operator) if
\[
-\Delta f=\Lambda f\,,
\]
where $\Lambda\geq 0$ is the corresponding eigenvalue. We shall denote by $\mathcal{E}_\Lambda$ the vector space of eigenfunctions with eigenvalue $\Lambda$. It is well known that the spectrum of the Laplacian on $\mathbb S^2$ and $\mathbb T^2$ is given by $\{\Lambda=k(k+1):k\in\mathbb N\}$ and $\{\Lambda=|k|^2:k\in \mathbb Z^2\}$, respectively.

We recall that the \textit{nodal set} $\mathcal N(f)$ of an eigenfunction $f$ is the set of points where $f$ vanishes. A nodal set is \textit{regular} if $0$ is a regular value of $f$, i.e., $\nabla f$ does not vanish at any point of $\mathcal N(f)$.

Our first result in this subsection is that for any eigenvalue $\Lambda>0$, the nodal set of a generic eigenfunction in $\mathbb S^2$ or $\mathbb T^2$ is regular (and nonempty because any nonconstant eigenfunction has zero mean). Given an integer $r>0$, by generic we mean that there exists an open and dense set of eigenfunctions in $\mathcal E_\Lambda$ (in the $C^r$ topology) that satisfies the aforementioned property. Since the result holds on any dimension $n\geq 2$, we state it in its full generality:

\begin{pr} \label{density}
For each eigenvalue $\Lambda>0$ of the Laplacian on the round sphere $\Ss^n$ or on the flat torus $\Tt^n$, a generic eigenfunction in $\mathcal{E}_\Lambda$ has a regular nodal set.
\end{pr}

\begin{proof}
We first recall the following parametric transversality theorem~\cite[Theorem 6.35]{lee}: let $N$ and $M$ be two smooth manifolds, $X \subset M$ an embedded submanifold, and $\{F_s\}_{s\in S}$ a smooth family of maps $F_s:N\to M$ with $S$ a smooth manifold of parameters. If the map $F: N \times S \to M$, $F(\cdot,s):=F_s(\cdot)$, is transverse to $X$, then for almost every $s \in S$, the map $F_s: N \to M$ is transverse to $X$.

For our purposes consider $N=\Ss^n$ or $\Tt^n$, $S=\mathcal E_\Lambda \cong \RR^m$ (where $m$ is the multiplicity of the eigenvalue $\Lambda$), $M=\RR$ and $X=\{0\}$. Let $\{f_j\}_{j=1,\dots, m}$ be an orthonormal basis of $\mathcal{E}_\Lambda$ and define:
$$
F: N \times \RR^m \to \RR\,, \quad F(x, a)=a_1 f_1(x) + \dots + a_m f_m(x)\,,
$$
with $x\in N$ and $a=(a_1,\dots, a_m)\in \mathbb R^m$. Let us now check that $F$ is transverse to $X=\{0\}$, that is $\nabla_{x,a} F$ does not vanish at the points $(x,a)$ for which $F(x,a)=0$. Here $\nabla_{x,a}$ denotes the gradient computed with respect to the variables $x\in N$ and $a\in \mathbb R^m$. A straightforward computation shows that if $\nabla_{x,a}F|_{(x_0,a_0)}=0$ then $f_1(x_0)=\dots =f_m(x_0)=0$, i.e., all the eigenfunctions of the basis vanish at the same point $x_0\in N$, and hence any $\Lambda$-eigenfunction.

Let us now consider a new orthonormal basis of the eigenspace $\mathcal{E}_\Lambda$ given by $\{\widetilde{f}_j\}_{j=1,\dots, m}$, where
$$\widetilde{f}_j := f_j \circ g^{-1}\,,$$
with $g\in \mathrm{Isom}(N)$ an isometry of $N$; the fact that this is an orthonormal basis is obvious because an isometry preserves the eigenspaces of the Laplacian and the scalar product. At the point $g \cdot x_0$ we have $\widetilde{f}_1(g \cdot x_0) = \dots = \widetilde{f}_m (g \cdot x_0)=0$, and therefore any eigenfunction $f\in \mathcal{E}_\Lambda$ vanishes at the point $g \cdot x_0$, i.e., $f(g \cdot x_0)=0$. Since, on $\Ss^n$ or $\Tt^n$, the action $\cdot$ of the isometry group is transitive, it follows that any $\Lambda$-eigenfunction $f$ vanishes on the whole $N$, which is a contradiction. Therefore $F$ has no critical points in $N\times \mathbb R^m$, so in particular it is transverse to $X=\{0\}$.

By the parametric transversality theorem stated above, $F_a$ is transverse to $X=\{0\}$ for almost all values of $a\in \RR^m$, which is equivalent to the fact that the nodal set of $F_a$ is regular for almost all (and hence for a dense set of) values of $a$. As the condition of being regular is open in the $C^r$ topology ($r\geq 1$), we conclude that there is an open and dense set of $\Lambda$-eigenfunctions whose nodal set is regular, as we wanted to show.
\end{proof}

The second main result of this subsection establishes the existence of eigenfunctions on $\mathbb T^2$ with regular nodal sets having a contractible connected component. By a contractible component we mean that $\mathbb T^2\backslash \mathcal N(f)$ has a component diffeomorphic to a disk. The proof makes use of a remarkable inverse localization theorem proved in~\cite{ept0,ept} which allows us to transplant the nodal set of a monochromatic wave in $\mathbb R^2$ into the nodal set of an eigenfunction in $\mathbb T^2$ with high eigenvalue.

\begin{pr} \label{nodalcirc}
There exists an infinite sequence of eigenvalues $\{\Lambda_\ell\}_{\ell \in \NN^*}$ and corresponding eigenfunctions $f_\ell$ of the Laplacian on $\Tt^2$ such that, for each $\ell$, the nodal set of $f_\ell$ is regular, disconnected, and contains a contractible connected component.
\end{pr}

\begin{proof} Let $(r, \theta)$ denote the polar coordinates in $\RR^2$. It is well known that the analytic function $w:=J_0(r)$ is a solution of the Helmholtz equation $\Delta w + w =0$ in $\mathbb R^2$ (i.e., a monochromatic wave). By definition, the nodal set of $w$ consists of concentric circles of radii given by the zeros of the Bessel function $J_0$: $0 < \alpha_1^0 < \alpha_2^0 < \dots$. We denote the disk in $\mathbb R^2$ of radius $10$ by $B$; the function $w$ has three components of its nodal set contained in $B$.

For any positive integer $r$, any $\delta>0$, and any point $p\in\mathbb T^2$, the inverse localization theorem~\cite[Theorem 3.2]{ept} implies that there exists an infinite sequence of natural numbers $\{N_\ell\}_{\ell \in \NN^*}$ and eigenfunctions $\widetilde f_\ell$ of the Laplacian on $\Tt^2$ with eigenvalue $\Lambda_{\ell} := N_\ell^2$ such that:
\begin{equation}\label{approx}
\left\Vert  w(\cdot) - \widetilde f_\ell \circ \Psi\left(\frac{\cdot}{N_\ell}\right)\right\Vert_{C^r(B)} \leq \delta\,,
\end{equation}
where $\Psi^{-1}$ is a patch of normal geodesic coordinates $\Psi^{-1}: \mathbb{B}_\rho \to B_\rho$ centered at $p$, from the ball $\mathbb{B}_\rho\subset \mathbb T^2$ of radius $\rho$ to the Euclidean ball $B_\rho$ of radius $\rho$.

Since the nodal set of $w$ is regular, Equation~\eqref{approx} and Thom's isotopy theorem imply that $\widetilde f_\ell$ has a nodal set with three components isotopic to three concentric circles in the ball of radius $10/N_\ell$ centered at $p\in \mathbb T^2$ (the isotopy being close to the identity). Although this part of the nodal set of $\widetilde f_\ell$ is regular, it may happen that the whole nodal set of $\widetilde f_\ell$ contains some critical points.

Accordingly, using Proposition~\ref{density}, we can take a $\Lambda_\ell$-eigenfunction whose nodal set is regular as close to $\widetilde f_\ell$ as one wishes. Therefore one may assume that $\widetilde f_\ell$ is $\varepsilon_\ell$-close in the $C^r$ norm to an eigenfunction $f_\ell$ whose nodal set is regular. Since we obviously have $\left\Vert \widetilde f_\ell \circ \Psi\left(\frac{\cdot}{N_\ell}\right) -  f_\ell \circ \Psi\left(\frac{\cdot}{N_\ell}\right)\right\Vert_{C^r(B)} \leq \varepsilon_\ell$ (using that $N_\ell^{-j}\leq 1$ if $j\in\{0,1,\cdots,r\}$), we conclude that:
\begin{equation*}
\begin{split}
\left\Vert  w -f_\ell \circ \Psi\left(\frac{\cdot}{N_\ell}\right)\right\Vert_{C^r(B)} & \leq \left\Vert  w - \widetilde f_\ell \circ \Psi\left(\frac{\cdot}{N_\ell}\right)\right\Vert_{C^r(B)} + \left\Vert \widetilde f_\ell \circ \Psi\left(\frac{\cdot}{N_\ell}\right) -  f_\ell \circ \Psi\left(\frac{\cdot}{N_\ell}\right)\right\Vert_{C^r(B)}\\
& \leq \delta + \varepsilon_\ell < \tfrac{3}{2}\delta
\end{split}
\end{equation*}
by choosing $\varepsilon_\ell < \frac{1}{2}\delta$.

Therefore, if $\delta$ is small enough, Thom's isotopy theorem implies again that $f_\ell$ has a nodal set with three components isotopic to three concentric circles in the ball of radius $10/N_\ell$, thus completing the proof.
\end{proof}

\section{Overtwisted Beltrami fields on $\Ss^3$: proof of Theorem~\ref{maint}}\label{S:sphere}

For convenience, we shall describe the sphere $\Ss^3$ as the set of points $(z_1, z_2) \in \CC^2$ with $\abs{z_1}^2 + \abs{z_2}^2 =1$. Denoting $z_j= x_j + \ii y_j$, $j\in\{1,2\}$, the Hopf coordinates are defined as $(x_1, y_1, x_2, y_2)=(\cos s \, e^{\ii\phi_1}, \sin s \, e^{\ii\phi_2})$, $s \in [0, \pi/2]$, $\phi_{1,2} \in [0, 2\pi)$. In these coordinates, the induced round metric on $\Ss^3$ reads as
$$g=\dif s^2 + \cos ^2 s \, \dif \phi_1^2 + \sin ^2 s \, \dif \phi_2^2\,,$$
and the unit (anti-)Hopf vector fields, dual to the (anti-)standard contact forms $\eta$ and $\eta^\prime$, respectively, are
$$
R =\partial_{\phi_1} + \partial_{\phi_2}\,, \quad \quad
R^\prime = \partial_{\phi_1} - \partial_{\phi_2}\,,
$$
with
$$
\eta =\cos ^2 s\dif\phi_1 + \sin ^2 s\dif\phi_2\,, \quad \quad
\eta^\prime = \cos ^2 s\,\dif\phi_1 - \sin ^2 s\,\dif\phi_2\,.
$$

We choose the orientation on $\Ss^3$ such that we have $\cu R = 2R$ and $\cu R^\prime = -2R^\prime$.
Therefore $R, R'$ are (nonvanishing) curl eigenfields (with eigenvalue $2$ and $-2$ respectively), and the induced contact structures are tight. In fact, notice that $R$ and $R'$ (and hence $\eta$ and $\eta^\prime$) are related via an orientation reversing contactomorphism.

It is well-known that the spectrum of the curl operator on $\Ss^3$ is given by
$$\{\lambda = \pm(k+2), k\in \NN \}\,.$$
In what follows, let us focus on the positive part of the spectrum, the analysis of the negative part being completely analogous (see Remark~\ref{R:negative} below).

To write the components of a vector field on $\mathbb S^3$ we shall use the standard (positively oriented) orthonormal global frame of Hopf vector fields $\{R,X_1,X_2\}$, where these fields have the explicit expressions
\begin{equation} \label{stdframe}
\begin{split}
R     & = - y_1 \partial_{x_1} + x_1 \partial_{y_1} - y_2 \partial_{x_2} + x_2 \partial_{y_2}\,,\\
X_1 & =-x_2 \partial_{x_1} + y_2 \partial_{y_1} + x_1 \partial_{x_2} - y_1 \partial_{y_2}\,,\\
X_2 & =-y_2 \partial_{x_1} - x_2 \partial_{y_1} + y_1 \partial_{x_2} + x_1 \partial_{y_2}\,.
\end{split}
\end{equation}
Notice that $X_1$ and $X_2$ are given by the action of an isometry of $\mathbb S^3$ on the field $R$ (the isometries being given by $\phi_1(x_1,y_1,x_2,y_2)=(x_2,-x_1,y_1,-y_2)$ and $\phi_2(x_1,y_1,x_2,y_2)=(y_1,x_2,-y_2,x_1)$).

Accordingly, an eigenfield $V$ of curl with eigenvalue $\lambda$ can be written as:
$$
V=fR+f_1X_1+f_2X_2
$$
for some functions $f,f_1,f_2$. The following key result was proved in~\cite[Proposition~1]{pss} (see also~\cite{kom}):
\begin{pr}\label{P:comps}
The functions $f,f_1,f_2$ are eigenfunctions of the Laplacian on $\mathbb S^3$ with eigenvalue $\lambda(\lambda-2)$.
\end{pr}

Let us first prove the first claim in Theorem~\ref{maint}, i.e., that any nonvanishing curl eigenfield on $\mathbb S^3$ has even eigenvalue.

\begin{lm}\label{L:odd}
Any curl eigenfield on $\mathbb S^3$ with odd eigenvalue has a nonempty zero set.
\end{lm}
\begin{proof}
Let $V$ be a curl eigenfield with odd eigenvalue $\lambda$. By Proposition~\ref{P:comps}, the components $f,f_1,f_2$ of $V$ are eigenfunctions of the Laplacian with eigenvalue $\lambda(\lambda -2)$. Therefore $f,f_1,f_2$ are the restriction on $\mathbb S^3$ of homogeneous harmonic polynomials on $\RR^4$ of degree $\lambda-2$; the common degree of these polynomials is then odd by our hypothesis. Therefore, Borsuk-Ulam theorem~\cite{borsuk} implies that $f,f_1,f_2$ have a common zero, which is a zero of the Beltrami field $V$, as we wanted to prove.
\end{proof}

To prove the second part of Theorem~\ref{maint}, let us now construct an explicit nonvanishing curl eigenfield for every even eigenvalue $\lambda\geq 4$. To achieve this we use the ansatz proposed in~\cite[Example 4.5]{kkpspaper} to construct steady Euler flows on $\mathbb S^3$:
\begin{equation}\label{kkps}
V = F(\cos^2 s) R + G(\cos ^ 2 s) R^\prime\,,
\end{equation}
where $F:\mathbb R\to \mathbb R$ and $G:\mathbb R\to \mathbb R$ are smooth functions that will be fixed later.

A straightforward computation shows (see~\cite[Example 4.5]{kkpspaper}) that $\cu V$ is given by
$$
\cu V=[(2z-1)F' +2F+G']R-[(2z-1)G' +2G+F']R^\prime\,,
$$
where we are using the notation $z:=\cos^2 s$. Using this expression, it then follows that a vector field on $\Ss^3$ of the form~\eqref{kkps} is a curl eigenfield with eigenvalue $\lambda$ if and only if
\begin{equation}\label{rotsyst}
\left\{
\begin{array}{lll}
(2z-1)F' +2F+G' =\lambda F\,, \\[2mm]
(2z-1)G' +2G+F' =-\lambda G\,.
\end{array}
\right.
\end{equation}

These equations imply that the function $F$ is a solution of the following hypergeometric equation:
$$
4z(z-1)F'' +8(2z-1)F' - (\lambda+ 4)(\lambda -2)F=0\,,
$$
which admits the analytic solution (in terms of a hypergeometric function) $F(z) =  _{2}\!\!F_1(1-\frac{\lambda}{2}, 2+\frac{\lambda}{2}, 2; \  z)$; this is the unique solution (up to a constant factor) which is bounded at $z=0$. The second component $G$ is determined by $F$ via the relation
$$G(z)=\frac{1}{\lambda + 2} (4 z (z - 1) F'(z) - (\lambda - 2) (2 z - 1) F(z))\,.$$

After some easy computations, we eventually obtain, for even eigenvalue $\lambda = 2m$, $m\geq 2$, the expressions (we introduce the subscript $m$ to emphasize the dependence of the functions $F$ and $G$ with the eigenvalue)
$$
F_m(z) =  _{2}\!\!F_1(1-m, 2+m, 2; \  z)\,, \quad
G_m(z) = \frac{m-1}{m+1} \ _{2}F_1(2-m, 1+m, 2; \  z)\,.
$$
These hypergeometric functions can be written in terms of two consecutive members of the family of orthogonal Jacobi polynomials $\{P_m^{(1,1)}\}_{m\geq 2}$ (cf. \cite[Definition 2.5.1]{specfunc}), of degree $m-1$ and $m-2$, respectively, that is
\begin{equation}\label{jac}
F_m(z)=\frac{1}{m}P_{m-1}^{(1,1)}(1-2z)\,, \qquad G_m(z)=\frac{1}{m+1}P_{m-2}^{(1,1)}(1-2z)\,.
\end{equation}
In fact, these functions actually belong to a special class of Jacobi polynomials, known as \textit{ultraspherical} or \textit{Gegenbauer polynomials} \footnote{To fix conventions, we take as definition \cite[p. 302]{specfunc} of Gegenbauer polynomials $P_n^{(\lambda)}$ (denoted also $C_n^{(\lambda)}$) the coefficients of the generating function $(1 - 2 t x + t^2)^{-\lambda}=\sum_{n=0}^\infty P_n^{(\lambda)}(x)t^n$. This convention is also used in Wolfram Mathematica, the software we have used for checking our computations.}, and we have
\begin{equation}\label{jac2}
F_m(z)=\frac{P_{m-1}^{(\frac{3}{2})}(1-2z)}{P_{m-1}^{(\frac{3}{2})}(1)}\,, \qquad G_m(z)=\frac{m-1}{m+1} \cdot \frac{P_{m-2}^{(\frac{3}{2})}(1-2z)}{P_{m-2}^{(\frac{3}{2})}(1)}\,.
\end{equation}

Accordingly, the zeros of $F_m$ and $G_m$ separate each other (or in other words, they interlace) and, in particular, $F_m$ and $G_m$ cannot have common zeros (see e.g.~\cite[Theorem 5.4.2]{specfunc}). Therefore, curl eigenfields of the form~\eqref{kkps}, associated to even eigenvalues, are nowhere vanishing. Indeed, on the complement of the Hopf link $\{s=0\}\cup \{s=\pi/2\}$, the vector fields $R$ and $R^\prime$ are linearly independent, so $V$ in~\eqref{kkps} cannot be zero because the functions $F_m$ and $G_m$ do not vanish simultaneously. On the Hopf link, $R$ and $R^\prime$ are collinear; a straightforward computation shows that
$$
V|_{s=0}=\frac{2(-1)^{m+1}}{m+1}\,\partial_{\phi_1} \qquad V|_{s=\pi/2}=\frac{2}{m+1}\,\partial_{\phi_2}\,,
$$
thus following that $V$ does not vanish on the Hopf link either.

Summarizing, we have constructed nonvanishing curl eigenfields of the form~\eqref{kkps} (and coefficients given by~\eqref{jac}) with eigenvalues $\lambda=2m$, $m\geq2$. As explained in Subsection~\ref{SS:Belta}, the dual $1$-forms of these fields are contact forms, so for each (positive) even eigenvalue of curl on
$\mathbb S^3$, there is a curl eigenfield that defines a contact structure. The following proposition shows that all these contact structures are overtwisted. The proof exploits the fact that these contact forms are $\Ss^1$-invariant (see also~\cite{kom} for other applications of this idea in the context of curl eigenfields).

\begin{pr}\label{P:OT}
The contact structures associated with the curl eigenfields~\eqref{kkps} with coefficients given by~\eqref{jac} are overtwisted.
\end{pr}
\begin{proof}
To emphasize the dependence with the eigenvalue, let us denote the curl eigenfields by $V_m$ ($\cu V_m=2m V_m$). They are Reeb fields (up to rescaling) for the contact forms on $\Ss^3$ given by their dual 1-forms $\alpha_m=V_m^\flat$. Moreover, these contact forms are $\Ss^1$-invariant with respect to the canonical $\Ss^1$-action on $\Ss^3$ whose infinitesimal generator is the Hopf field $R$, that is:
$$
\mathcal{L}_R \alpha_m =0, \qquad m\in\{2,3,4,\dots\}\,.
$$
The characteristic surface $\Gamma_R$ of the contact structure defined by the kernel of $\alpha_m$ is defined~\cite{gir} as
$$\Gamma_R:=\{p \in \Ss^3 : R \ \text{tangent to the contact distribution} \ \ker  \alpha_m \ \text{at} \  p\}\,.$$
In our case, it can be described using the roots of a polynomial equation in the variable $z=\cos^2 s$ writing the condition $R\cdot V_m=0$ (using the round metric) at a point $p=(\cos s \, e^{\ii\phi_1}, \sin s \, e^{\ii\phi_2})\in \Ss^3$, that is
\begin{equation}\label{charactsurf}
F_m(z)+(2z-1)G_m(z)=0\,.
\end{equation}
The zeros of this equation are given by different values of $z\in(0,1)$ (it is immediate to check that $z=0$ and $z=1$ are not solutions), and hence of the Hopf radius $s$, so they correspond to toroidal surfaces on $\mathbb S^3$. By Giroux characterization of tight $\Ss^1$-invariant contact structures~\cite[Proposition 4.1b]{gir}, the contact structure $\alpha_m$ on $\mathbb S^3$ is tight if and only $\pi(\Gamma_R) =\emptyset$, where $\pi:\Ss^3 \to \Ss^2$ is the Hopf fibration, seen as a circle bundle of Euler number $-1$, whose fibres are tangent to $R$.

Lemma~\ref{L:zeros} below shows that Equation~\eqref{charactsurf} has at least one solution, which corresponds to a torus $s=s_0$, $s_0\in(0,1)$, which is a component of the characteristic surface $\Gamma_R$. Obviously, this torus projects onto a circle in $\Ss^2$ through the Hopf fibration, so $\pi(\Gamma_R) \neq \emptyset$, thus proving that $\alpha_m$ corresponds to an overtwisted contact structure on $\Ss^3$, for any $m\in\{2,3,4,\dots\}$, as we wanted to prove.
\end{proof}

\begin{lm}\label{L:zeros}
Equation~\eqref{charactsurf} has at least one solution $z\in (0,1)$.
\end{lm}

\begin{proof}
Assume $m\geq 5$. The function $G_m$ is, by definition, a polynomial of degree $m-2\geq 3$. Since the zeros of the Jacobi polynomials lie in the interval $[-1,1]$, and are even or odd depending on the parity of $m$, it follows that $G_m$ has (at least) two consecutive zeros $0\leq z_1<z_2\leq 1$. Now, the polynomial $\mathcal{P}_m(z):=F_m(z)+(2z-1)G_m(z)$ evaluated at these points gives $\mathcal P_m(z_k)=F_m(z_k)$, $k=1,2$. Since the zeros of $G_m$ and $F_m$ interlace, we have that $F_m(z_1)F_m(z_2) < 0$, which is equivalent to $\mathcal{P}_m(z_1)\mathcal{P}_m(z_2) < 0$, thus showing that $\mathcal{P}_m$ has a zero between $z_1$ and $z_2$. The remaining cases are $m=2$ (when $\mathcal{P}_2(z)=\tfrac{2}{3}- \tfrac{4}{3}z$), $m=3$
(when $\mathcal{P}_3(z)=\tfrac{1}{2} -3z + 3z^2$), and $m=4$ (when $\mathcal P_4(z)=\tfrac{2}{5}-\tfrac{24}{5}z +12z^2-8z^3$), for which the statement is obviously true.
\end{proof}

\begin{ex}\label{concretexi}
The nonvanishing curl eigenfields $V_m$ constructed above can be expressed in Hopf coordinates as follows (for $m=2,3$):
\begin{equation*}
\begin{split}
V_2 &= -\tfrac{1}{3} (3 \cos 2 s - 1)\partial_{\phi_1} -  \tfrac{1}{3} (3 \cos 2 s  +  1)\partial_{\phi_2}\,, \\
V_3 &= (\tfrac{3}{2} - 6 \cos^2 s + 5 \cos^4 s)\partial_{\phi_1} +
 (\tfrac{1}{2} - 4 \cos^2 s + 5 \cos^4 s)\partial_{\phi_2}\,.
\end{split}
\end{equation*}
In particular, $R\cdot V_2=-\tfrac{2}{3}\cos 2s$ and $R\cdot V_3=\tfrac{1}{8} (3 \cos 4 s + 1)$ are the left hand sides of Equation~\eqref{charactsurf}, which describes the respective characteristic surfaces.
\end{ex}

To conclude the proof of Theorem~\ref{maint}, let us now show that the Hopf invariant of the overtwisted contact structures given by the 1-forms $\alpha_m$ dual to $V_m$, $m\geq 2$, is $-1$ or $0$, depending on whether the value of $m$ is even or odd, respectively. We recall that this Hopf invariant is given by the Hopf invariant of the map $\varphi_m:\mathbb S^3\to\mathbb S^2$ defined as follows. Write the vector field $V_m$ as $V_m = f R + f_1 X_1 + f_2 X_2$, with respect to the standard positively oriented orthonormal global frame $\{R, X_1, X_2 \}$ on $\Ss^3$. Then the map associated to (the contact structure defined by) $V_m$ is
$$\varphi_m:\Ss^3 \to \Ss^2\,, \qquad \varphi_m(p)=\frac{1}{\sqrt{f(p)^2 + f_1(p)^2 + f_2(p)^2}}\big(f(p), f_1(p), f_2(p)\big)\,.$$

\begin{re}\label{rem_duf}
Given a pair of nonvanishing vector fields $X$ and $V$ on $\Ss^3$, following~\cite{duf} one can define a relative homotopy invariant $H_X(V)\in \ZZ$, so that $X$ and $V$ are homotopic if and only if $H_X(V)=0$, cf.~\cite[Lemma 10]{duf}. It is easy to check that the Hopf invariant of $V$ as defined above coincides with $H_R(V)$ (after fixing the global frame $\{R, X_1, X_2 \}$ on $\Ss^3$).
\end{re}

Instrumental to the proof of Proposition~\ref{P:Hopf} below is the following lemma, which is standard, but we include a proof for the sake of completeness.

\begin{lm}\label{L:comp}
The Hopf invariants of the Hopf and anti-Hopf fields are $0$ and $-1$, respectively.
\end{lm}

\begin{proof}
For the Hopf field $R$ the claim is obvious (the associated map $\varphi_R$ is a constant). The anti-Hopf field reads in terms of the global frame $\{R,X_1,X_2\}$ as:
$$R^\prime = \cos 2s \, R + \sin 2s \sin(\phi_1 + \phi_2) X_1 - \sin 2s \cos(\phi_1 + \phi_2) X_2\,.$$
Then the associated map is
$$\varphi_{R^\prime}(\cos s \, e^{\ii\phi_1}, \sin s \, e^{\ii\phi_2})=\big(\cos 2s, \ \sin 2s \, e^{\ii (\phi_1 + \phi_2 - \frac{\pi}{2})}\big)\,.$$
The orientation on $\Ss^3$ is fixed by the requirement that $\{R, X_1, X_2\}$ is a positive orthonormal frame (this is the usual convention in Contact Topology for which $\eta\wedge \dif \eta>0$), so the volume form in Hopf coordinates reads as $-\sin s \cos s \dif s \wedge \dif \phi_1 \wedge \dif \phi_2$. The area form on $\Ss^2$ is $\Omega = x dy \wedge dz + y dz \wedge dx + z dx \wedge dy$, so it is easy to check that $\varphi_{R^\prime}^* \Omega = -2 \dif(\cos^2 s \dif \phi_1 -\sin^2 s \dif \phi_2) =:\dif \mathcal{A}$ (so $\mathcal A=-2\eta^\prime$ up to a closed $1$-form). A straightforward application of Whitehead's formula gives
$$\text{Hopf invariant}(\varphi)=\frac{1}{16\pi^2}\int_{\Ss^3} \mathcal{A} \wedge \dif \mathcal{A} =-1\,,$$
thus showing that the Hopf invariant of (the contact structure associated to)  $R^\prime$ is $-1$.
\end{proof}

\begin{pr}\label{P:Hopf}
The homotopy classes of the contact structures associated to $V_m$ have Hopf invariant (or Dufraine's invariant $H_R(V_m)$, cf. Remark~\ref{rem_duf})
$$
\text{Hopf invariant}=H_R(V_m)=\frac{(-1)^{m+1}-1}{2}\,.
$$
\end{pr}
\begin{proof}
For each integer $m\geq 2$, we consider the following smooth family of vector fields on $\Ss^3$, indexed by $t\in [0,1]$:
\begin{equation*}
V^{t}_{m} = \frac{1}{m^2}P_{m-1}^{(1,1)}(1-t)P_{m-1}^{(1,1)}((t-1)\cos 2s)R +
\frac{1}{(m+1)^2}P_{m}^{(1,1)}(1-t)P_{m-2}^{(1,1)}((t-1)\cos 2s)R^\prime\,.
\end{equation*}
We claim that
\begin{itemize}
\item[(i)] $V^{t}_{m}$ does not vanish on $\Ss^3$, for any integer $m\geq 2$ and any $t\in [0,1]$.

\item[(ii)] $V^{0}_{m} = V_m$.

\item[(iii)] Up to a multiplicative constant, $V^{1}_{m}$ is equal to the Hopf field $R$ if $m$ is odd, and to the anti-Hopf field $R^\prime$ if $m$ is even.
\end{itemize}
From these claims we deduce that $V_m$ is homotopic (through nonvanishing vector fields) to either $R$ or $R^\prime$, depending on whether $m$ is odd or even, respectively, so the associated map $\varphi_m$ has Hopf invariant $0$ when $m$ is odd and $-1$ when $m$ is even (cf. Lemma~\ref{L:comp}), and the proposition follows.

In order to prove the claim (i), let us recall that the vector fields $R$ and $R^\prime$ are linearly independent at any point of $\mathbb S^3$ except on the Hopf link, where we have that $R=R^\prime$ if $s=0$ and $R=-R^\prime$ if $s=\frac{\pi}{2}$. Accordingly, since consecutive Jacobi polynomials do not vanish at the same point, we have that away from the Hopf link, the vector field $V^{t}_{m}$ is clearly nonvanishing for all $m$ and $t\in(0,1)$, and on the Hopf link we have
\begin{equation*}
\begin{split}
V^{t}_{m} &= (-1)^{m-1}\left(\frac{1}{m^2}[P_{m-1}^{(1,1)}(1-t)]^2 -\frac{1}{(m+1)^2}P_{m}^{(1,1)}(1-t)P_{m-2}^{(1,1)}(1-t)\right)R\,, \quad \text{at} \ s=0\,,\\
V^{t}_{m} &=\left(\frac{1}{m^2}[P_{m-1}^{(1,1)}(1-t)]^2 -\frac{1}{(m+1)^2}P_{m}^{(1,1)}(1-t)P_{m-2}^{(1,1)}(1-t)\right)R\,, \quad \text{at} \ s=\tfrac{\pi}{2}\,.
\end{split}
\end{equation*}
Let us now show that the field $V^{t}_{m}$ on $s=0$ and $s=\pi/2$ does not vanish for all $t\in(0,1)$. Indeed, using Tur\'an's inequality~\cite{sz}, in terms of Jacobi polynomials,
\begin{equation} \label{turan}
[P_{m-1}^{(1,1)}(w)]^2 > \frac{m^2}{m^2-1}P_{m}^{(1,1)}(w)P_{m-2}^{(1,1)}(w)>P_{m}^{(1,1)}(w)P_{m-2}^{(1,1)}(w)\,, \qquad \abs{w} < 1\,,
\end{equation}
it is immediate to check that
$$\frac{1}{m^2}[P_{m-1}^{(1,1)}(1-t)]^2 > \frac{1}{(m+1)^2}P_{m}^{(1,1)}(1-t)P_{m-2}^{(1,1)}(1-t)\,,$$
for all $t\in(0,1)$, and hence $V^t_m$ does not vanish on the Hopf link.

Claim (ii) is an immediate consequence of the identity~\cite{specfunc}
$$P_{m}^{{(1, 1)}}(1)=m+1\,,$$
and that $1-2\cos^2s=-\cos 2s$.

Finally, to prove the claim (iii), we first write
$$V^1_{m} = \frac{1}{m^2}[P_{m-1}^{(1,1)}(0)]^2 R +
\frac{1}{(m+1)^2}P_{m}^{(1,1)}(0)P_{m-2}^{(1,1)}(0)R^\prime\,,$$
and observe that $P_{m}^{{(1, 1)}}(0) \neq 0$ if and only if $m$ is even, since for Jacobi polynomials it is well known~\cite{specfunc} that they have a definite parity, i.e., $P_{m}^{{(\alpha ,\beta )}}(w) =(-1)^m P_{m}^{{(\beta ,\alpha )}}(-w)$, and consecutive Jacobi polynomials have different zeros.
\end{proof}

\begin{re}
An alternative way of computing the homotopy classes of the contact structures is via the associated map $\varphi_m$ between spheres. Writing the curl eigenfield $V_m$ in terms of the (positively oriented) orthonormal global frame $\{R,X_1,X_2\}$, i.e. $V_m = f R + f_1 X_1 + f_2 X_2$, the map $\varphi_m:\mathbb S^3\to \mathbb S^2$ is of the following form:
\begin{equation*}
\begin{split}
\varphi_m(\cos s \, e^{\ii\phi_1}, \ \sin s \, e^{\ii\phi_2})=\left(\frac{F_m+\cos 2s \, G_m}{H_m}, \ \frac{\sin 2s \, G_m}{H_m} \, e^{\ii(\phi_1+\phi_2 - \frac{\pi}{2})}\right),
\end{split}
\end{equation*}
where $H_m:= \sqrt{F_m^2+2\cos 2s \, F_mG_m+G_m^2}$, the functions $F_m$ and $G_m$ evaluated at $\cos^2 s$. For example, for $m=2$ we obtain
$$\varphi_m(\cos s \, e^{\ii\phi_1}, \ \sin s \, e^{\ii\phi_2})=\left(-\frac{2 \sqrt{2} \cos 2s}{\sqrt{5 + 3 \cos 4s}},\ \frac{\sqrt{2} \sin 2s}{\sqrt{5 + 3 \cos 4s}} e^{\ii(\phi_1+\phi_2 - \frac{\pi}{2})}\right)\,,$$
and the Hopf invariant can be computed to be $-1$ using Whitehead's integral formula. In principle, the same procedure can be followed for higher eigenvalues, however the formulas become more and more cumbersome (which makes impossible in practice to evaluate the corresponding Whitehead's integral).
\end{re}

\begin{re} Using Tur\'an's inequality~\eqref{turan} it is possible to show that, for consecutive even eigenvalues, $\lambda = 2m$ and $\lambda=2m+2$, $m\geq 2$, the curl eigenfields $V_m$ found above are collinear only at $\{s=\pi/2\}$ ($C_+$ is one of the components of the Hopf link) and anti-collinear only at $\{s=0\}$ ($C_-$ is the other component of the Hopf link). Indeed the proportionality (for $s \in (0,\pi/2)$) between the curl eigenfields for $\lambda = 2m$ and $\lambda=2m+2$, with coefficients $F_m, G_m$ and $F_{m+1}, G_{m+1}$, respectively, translates into the requirement that
\begin{align}
&\frac{m(m + 1)}{(m - 1)(m+2)}\big(F_{m}(z)\big)^2=F_{m-1}(z)F_{m+1}(z)\\ &\Leftrightarrow [P_{m-1}^{(1,1)}(1-2z)]^2 = \frac{m(m+2)}{(m+1)^2}P_{m}^{(1,1)}(1-2z)P_{m-2}^{(1,1)}(1-2z)
\end{align}
should hold for some $z\in (0,1)$, which contradicts~\eqref{turan} (because $\frac{m(m+2)}{(m+1)^2}<1$ for all $m\geq 2$). Moreover, a straightforward computation shows that on $s=\pi/2$ the fields are
\begin{equation*}
V_m=\frac{2}{m+1}R\,, \qquad V_{m+1}=\frac{2}{m+2}R\,,
\end{equation*}
and hence they are collinear, while on $s=0$ they are given by
\begin{equation*}
V_m=\frac{-2(-1)^m}{m+1}R\,, \qquad V_{m+1}=\frac{2(-1)^m}{m+2}R\,,
\end{equation*}
so they are anti-collinear.

Since Ref.~\cite[Lemma 23]{duf} implies that the linking number of the link formed by the collinear and anti-collinear curves $C_+,C_-$ of the vector fields $V_m$ and $V_{m+1}$ gives the distance in homotopy classes of both fields, we conclude that these homotopy classes are adjacent (because $|\mathrm{link}(C_+, C_-)|=1$), which is consistent with Proposition~\ref{P:Hopf} (although it does not compute exactly the homotopy class).
\end{re}

\begin{re}[Negative eigenvalues]\label{R:negative}
For the case of negative eigenvalues $\lambda\in\{-2,-3,-4,\dots\}$, Lemma~\ref{L:odd} also holds (i.e. if $|\lambda|$ is odd, any curl eigenfield vanishes at some point). The construction presented in Equation~\eqref{kkps} for $\lambda=-2m$, gives a curl eigenfield $V_m^\prime$ of the form
$$
V_m^\prime=G_mR+F_mR^\prime\,,
$$
where $F_m$ and $G_m$ are the same functions obtained in Equation~\eqref{jac}. Arguing now exactly as in the proofs of Propositions~\ref{P:OT} and~\ref{P:Hopf}, we conclude that the corresponding contact structures are overtwisted with Hopf class equal to $-1$ or $0$ depending on whether $m$ is odd or even, respectively. An alternative way of proceeding is noticing that the Hopf field $R$ and the anti-Hopf field $R^\prime$ are related by the orientation-reversing diffeomorphism $(x_1,y_1,x_2,y_2)\to (x_1,y_1,x_2,-y_2)$ restricted to $\mathbb S^3$.
\end{re}

\section{Beltrami fields with constant norm on $\Ss^3$: proof of Proposition~\ref{maint2}}\label{S:const}

In this section we provide a completely different proof of~\cite[Theorem A]{glu} (in the special case of curl eigenfields) claiming that the only contact forms on $\mathbb S^3$ whose compatible metric is the round one are those whose Reeb field is isometric (up to a constant proportionality factor) to the Hopf field $R$ (or the anti-Hopf field $R^\prime$ in the case we consider negative contact forms), cf. Proposition~\ref{maint2}. All along this section we shall use the notation introduced in Section~\ref{S:sphere} without further mention.

Let us first show that a curl eigenfield on $\mathbb S^3$ with eigenvalue $\lambda>2$ cannot have constant norm. Suppose the contrary, so let $V$ be a constant norm vector field, which can be rescaled so that $\abs{V}=1$, with $\cu V=\lambda V$. Lemma~\ref{L:odd} implies that $\lambda$ is an even number $\geq 4$. With respect to the standard orthonormal global frame, it decomposes as $V = f R + f_1 X_1 + f_2 X_2$, so the associated map $\varphi_V: \Ss^3 \to \Ss^2$ is given by
$$p\in \mathbb S^3 \mapsto (f, f_1, f_2)(p)\,.$$

As $f, f_1$ and $f_2$ are eigenfunctions of the Laplacian on $\Ss^3$ with eigenvalue $\lambda(\lambda -2)\geq 8$ (cf. Proposition~\ref{P:comps}), it follows that $\varphi_V$ is an eigenmap~\cite[Example 3.3.18]{BW}, i.e., it is a harmonic map with constant energy density $\abs{\dif \varphi_V}^2=\lambda(\lambda - 2)$. Moreover, for a Beltrami field on $\Ss^3$ we have
$$
V\cdot\nabla f=V\cdot \nabla(V\cdot R)=\nabla_VV\cdot R+\nabla_VR\cdot V=\frac12R\cdot\nabla(\abs{V}^2)=0\,,
$$
where in the third equality we have used the well known identity for Beltrami fields $\nabla_V V=\frac12 \nabla(\abs{V}^2)$, and that $R$ is a Killing vector field and hence $V\cdot \nabla_VR=0$. Arguing in the same way we obtain
$$V\cdot \nabla f_1=\frac{1}{2}X_1\cdot \nabla(\abs{V}^2)=0\,, \qquad V\cdot \nabla f_2=\frac{1}{2}X_2\cdot \nabla(\abs{V}^2)=0\,.$$
The components $f, f_1, f_2$ are then first integrals of $V$, and therefore
$$V \in \ker \dif \varphi_V\,.$$
Furthermore, $\varphi_V$ has geodesic fibres because $\nabla_V V=0$ (by the aforementioned identity).

For any harmonic map $\varphi:M\to N$ between two Riemannian manifolds we have the Weitzenb\"ock formula~\cite{ES}, in terms of a local orthonormal basis $\{e_i\}_{i=1}^{\text{dim }M}$:
\begin{equation}
\tfrac{1}{2}\Delta \abs{\dif \varphi}^2= \abs{\nabla \dif \varphi}^2 +\sum_i \langle \dif \varphi(\Ric^M e_i), \dif \varphi(e_i) \rangle -\sum_{i,j}\langle R^N (\dif \varphi(e_i), \dif \varphi(e_j))\dif \varphi(e_j), \ \dif \varphi(e_i) \rangle\,,
\end{equation}
which reads, in our case $M=\Ss^3$ and $N=\Ss^2$, as:
$$\tfrac{1}{2}\Delta \abs{\dif \varphi}^2 = \abs{\nabla \dif \varphi}^2 + 2\abs{\dif \varphi}^2 -2 \abs{\Lambda^2\dif \varphi}^2\,.$$
Here $\Lambda^2 \dif\varphi = \dif\varphi \wedge \dif\varphi$ is the extension of the exterior derivative to $2$-vectors. Its norm is the norm of the pullback (with $\varphi$) of the area form on $\mathbb S^2$.

Applying this to the map $\varphi_V$ with $\abs{\dif \varphi_V}^2=\lambda(\lambda - 2)> 0$ yields
$$0= \abs{\nabla \dif \varphi_V}^2 + 2\lambda(\lambda - 2) -2 \abs{\Lambda^2\dif \varphi_V}^2\,,$$
and therefore
$$
2 \abs{\Lambda^2\dif \varphi_V}^2\geq 2\lambda(\lambda - 2)>0\,,
$$
which implies that $\dif \varphi_V$ must have constant rank 2 (recall that the points where $\dif\varphi_V$ has rank smaller than $2$ correspond to zeros of the pullback of the area form on $\mathbb S^2$, and hence to zeros of $\abs{\Lambda^2\dif \varphi_V}^2$). This allows us to apply~\cite[Corollary 4]{baird} to deduce that $\varphi_V$ is a harmonic morphism and~\cite[Theorem 1]{eells} to deduce that $\varphi_V$ is isometrically equivalent to the Hopf fibration. Since $V \in \ker \dif \varphi_V$, and $\abs{V}=1$, this easily implies that $V$ is isometric to the Hopf field. Since isometries preserve curl eigenspaces, we would have that $\lambda=2$, which contradicts our assumption that $\lambda>2$.

To finish the proof of Proposition~\ref{maint2}, we observe that any curl eigenfield on $\Ss^3$ with $\lambda=2$ must have constant norm (as it is a linear combination with constant coefficients of the three vector fields $R, X_1, X_2$ giving the standard frame). Moreover, normalizing so that $\abs{V}=1$, and noticing that the action of the isometry group of $\mathbb S^3$ on the Hopf field $R$ gives a $2$-dimensional orbit diffeomorphic to $\mathbb S^2$ (see e.g.~\cite[Lemma~1]{Wi96}), we conclude that any eigenfield with eigenvalue $2$ is isometric to the Hopf field.

\begin{re}
An analogous proof holds for the case of negative eigenvalues $\lambda\leq -2$. In this case the eigenvalue must be $\lambda=-2$ and the curl eigenfield is isometric to the anti-Hopf field $R^\prime$.
\end{re}

\section{Nonvanishing Beltrami fields on $\Tt^3$: proof of Theorem~\ref{maint3}}\label{S:torus}

In this section we consider the torus $\Tt^3 = \RR^3/ (2\pi \ZZ)^3$ endowed with the flat metric. Recall that the spectrum of the curl operator is given by
$$\{\lambda=\pm\abs{k}:k \in \ZZ^3\}\,.$$

It is easy to check that for any nonzero vector $b \in \RR^3$, $b \perp k$, the vector field
\begin{equation}\label{genericbelt3}
V_k=\cos(k \cdot x)\, b + \frac{1}{\abs{k}}\sin(k \cdot x)\, b \times k
\end{equation}
is an eigenfield of $\cu$ with eigenvalue $\abs{k}$. A straightforward computation shows that $\abs{V_k}= \abs{b}$ (constant norm), so $V_k$ is a nonvanishing curl eigenfield, and hence it induces a contact structure on $\Tt^3$. We first show that all these contact structures are tight, and the associated map $\varphi_k$ is homotopically trivial. We recall that the map $\varphi_k:\mathbb T^3\to \mathbb S^2$ is defined as
$$
\varphi_k(x)=\frac{1}{\sqrt{A(x)^2 + B(x)^2 + C(x)^2}}\big(A(x), B(x), C(x)\big)\,,
$$
where $V_k=A\partial_{x_1}+B\partial_{x_2}+C\partial_{x_3}$ in the standard coordinates $x=(x_1,x_2,x_3)\in \RR^3/ (2\pi \ZZ)^3$.

\begin{lm}\label{L:torus}
The map $\varphi_k$ is homotopically trivial (in fact, it is a harmonic map with constant energy density). Moreover, for all $k\in\mathbb Z^3$, the contact structures associated with the curl eigenfields $V_k$ are tight.
\end{lm}
\begin{proof}
It is obvious that $k\cdot V_k=0$. It then follows that the image of the map $\varphi_k : \Tt^3 \to \Ss^2$ stays in the intersection of a plane with the $2$-sphere, and therefore $\varphi_k$ (and the contact structure induced by $V_k$) is homotopically trivial. Moreover, since the components of the vector field $V_k$ are eigenfunctions of the Laplacian with eigenvalue $\abs{k}^2$, $\varphi_k$ is an eigenmap~\cite[Example 3.3.18]{BW}, so harmonic with constant energy density.

Let us denote by $\eta_k$ the contact form dual to $V_k$. By definition, the flat metric is compatible with the contact structure $\eta_k$ on $\Tt^3$. Consider the obvious lift of $\eta_k$ to $\mathbb R^3$, which we will still denote by $\eta_k$, and assume without loss of generality that $\abs{b}=1$. Of course, the frame $\{ b\times k/|k|, b, k/|k|\}$ defines an orthonormal frame of $\RR^3$. Now take an orthogonal transformation of $\mathbb R^3$ that transforms this frame into the standard orthonormal frame $\{e_1,e_2,e_3\}$. After this transformation, the 1-form $\eta_k$ reads as:
$$
\cos(|k|x_3)\dif x_2 + \sin(|k|x_3) \dif x_1\,.
$$
Acting further with the diffeomorphism $(x_1,x_2,|k|x_3) \to (x_1,x_2,x_3)$ of $\mathbb R^3$, we obtain the $1$-form
$$\cos x_3 \dif x_2 + \sin x_3 \dif x_1\,,$$
which is contactomorphic to the tight standard contact structure in $\RR^3$ (\cite[Example 3.25]{geighand}). Summarizing, we have shown that the contact structure defined by $\eta_k$ lifted to $\mathbb R^3$ is contactomorphic to the tight standard contact structure, thus proving the universal tightness (and hence the tightness) of the contact structure $\eta_k$ in $\Tt^3$.
\end{proof}

The first part of Theorem~\ref{maint3} follows from Lemma~\ref{L:torus} and the following observation. In the particular case that $k=(0,0,m)$, $m\in\mathbb Z$, and $b=(0,1,0)$, we find the following standard family of contact structures on $\Tt^3$. We write it in terms of the $1$-forms dual to the vector fields $V_k$:
\begin{equation}\label{standardT3}
\eta_m = \sin(mx_3)\dif x_1 + \cos(mx_3)\dif x_2\,, \qquad m\in \ZZ\,,
\end{equation}
which corresponds to the integer part of the curl spectrum (i.e. $\ast \dif \eta_m = m \eta_m$). Of course, these contact forms $\eta_m$ are tight and homotopically trivial, but they belong to distinct contactomorphic classes (i.e., there is no contact diffeomorphism between $(\mathbb T^3,\zeta_n)$ and $(\mathbb T^3,\zeta_m)$ if $n\neq m$). In addition, any tight contact structure on $\Tt^3$ is contactomorphic to one of these $\eta_m$~\cite{kan}.

\begin{re}
We have not been able to classify the contactomorphic classes of the contact structures associated with all the curl eigenfields $V_k$ defined above. The following example settles a particular case with eigenvalue $\lambda=\sqrt{2}$. Consider the eigenfield $V_k$ with $k=(1,1,0)$ and $b=(0,0,1)$, actually its dual $1$-form $\eta$:
$$
\eta := -\frac{1}{\sqrt{2}} \sin(x_1 + x_2) \dif x_1 +
\frac{1}{\sqrt{2}} \sin(x_1 + x_2) \dif x_2 +
\cos(x_1 + x_2) \dif x_3\,.
$$
The pullback of $\eta$ by the $SL(3, \ZZ)$ diffeomorphism $\Psi(x_1,x_2,x_3)=(2x_1 - x_2, -x_1 + x_2, x_3)$ of $\mathbb T^3$ is
$$
\Psi^*\eta= \sqrt{2} \sin(x_1) \dif x_2 + \cos(x_1) \dif  x_3 - \frac{3}{\sqrt{2}}\sin(x_1) \dif x_1\,.
$$
Consider now the homotopy
$$
\alpha_t:=\sqrt{t+1} \sin(x_1) \dif x_2 + \cos(x_1) \dif  x_3 - \frac{3t}{\sqrt{2}}\sin(x_1) \dif x_1\,,
$$
$t\in [0,1]$, which is a continuous path between $\Psi^*\eta $ and the eigenform
$\sin(x_1) \dif x_2 + \cos(x_1) \dif  x_3$ which is isometrically related to $\eta_1$.
Since
$$
\alpha_t \wedge \dif \alpha_t = \sqrt{t+1} dx_1 \wedge dx_2 \wedge dx_3 > 0\,,
$$
we conclude that $\eta$ is contactomorphic to $\eta_1$.
\end{re}

To prove the second part of Theorem~\ref{maint3}, following~\cite{kom} we consider an $\Ss^1$-invariant ansatz for curl eigenfields on $\Tt^3$:
\begin{equation}\label{ansatzT3}
V=\frac{\partial f}{\partial x_2} \partial_{x_1} -\frac{\partial f}{\partial x_1} \partial_{x_2} +  \lambda f \ \partial_{x_3}\,,
\end{equation}
where $f\equiv f(x_1, x_2)$ is a function on $\Tt^2$. It is straightforward to check that $V$ is an eigenfield of $\cu$ with eigenvalue $\lambda$ if and only if $f$ is an eigenfunction of the Laplacian on $\mathbb T^2$ with eigenvalue $\Lambda=\lambda^2$  (see also~\cite[Theorem 3.4]{kom}).

If $f_\ell$, $\ell \in \mathbb N^*$, is an eigenfunction of the Laplacian on $\mathbb T^2$ with eigenvalue $\Lambda_\ell$ and regular nodal set, the previous discussion implies that the vector field $V_\ell$ of the form~\eqref{ansatzT3} (with $f\equiv f_\ell$) is a nonvanishing curl eigenfield of eigenvalue $\lambda_\ell:=\Lambda_\ell^{1/2}$ in $\mathbb T^3$. Assume also that the nodal set $\mathcal N_\ell$ of $f_\ell$ is disconnected and $\mathbb T^2\setminus \mathcal N_\ell$ has a component diffeomorphic to a disk for all $\ell$.

The contact form $\eta_\ell := V_\ell^\flat$ associated to $V_\ell$ is $\Ss^1$-invariant with respect to the $\Ss^1$-action on $\Tt^3$ whose infinitesimal generator is $Z:=\partial_{x_3}$. Defining the characteristic surface $\Gamma_Z^\ell$ as
\[
\Gamma_Z^\ell:=\{p \in \mathbb T^3 : Z \ \text{tangent to the contact distribution} \ \ker  \eta_\ell \ \text{at} \  p\}\,,
\]
Giroux characterization~\cite[Proposition 4.1a]{gir} implies that if $\Tt^2 \setminus \pi(\Gamma_{Z}^\ell)$ has a component diffeomorphic to a disk, the contact structure defined by $\eta_\ell$ is tight only if $\pi(\Gamma_{Z}^\ell)$ is connected. Here $\pi(x_1, x_2, x_3)=(x_1, x_2)$ is the projection onto $\Tt^2$. It is obvious that the projection of the characteristic surface $\pi(\Gamma_{Z}^\ell)$ is precisely the nodal set $\mathcal N_\ell$ of $f_\ell$, so by hypothesis, $\mathcal N_\ell$ is disconnected and $\mathbb T^2\setminus \mathcal N_\ell$ has a component diffeomorphic to a disk, thus implying that the contact structure defined by $\eta_\ell$ is overtwisted for all $\ell$. Since Proposition~\ref{nodalcirc} establishes the existence of an infinite sequence of eigenvalues $\Lambda_\ell$ and corresponding eigenfunctions $f_\ell$ satisfying the required assumptions, Theorem~\ref{maint3} follows.

\section{Examples and final remarks}\label{S:final}

\subsection{Existence of overtwisted curl eigenfields on $\mathbb S^3$}\label{SS:Raf}

As suggested to us by R. Komendarczyk (private communication), the existence of overtwisted contact structures on $\Ss^3$ associated with nonvanishing curl eigenfields can be established using an $\Ss^1$-invariant ansatz akin to the one used in Section~\ref{S:torus} for $\mathbb T^3$:
\begin{equation} \label{raf}
V = \lambda f R + X_2(f) X_1 - X_1(f) X_2\,,
\end{equation}
where we recall that $\{R,X_1,X_2\}$ is the orthonormal global frame of Hopf vector fields introduced in Section~\ref{S:sphere}. Here we assume that the function $f$ is a first integral of the Hopf field $R$. If $f$ is an eigenfunction of the Laplacian on $\mathbb S^3$ with eigenvalue $\lambda(\lambda-2)$ then $V$ is an eigenfield of $\cu$ with eigenvalue $\lambda$. Let us elaborate on this.

Indeed, following~\cite[Theorem 3.4]{kom} (see also~\cite[Equation (2.6)]{pss}), we notice that $R$ is the infinitesimal generator of the fibers of $\Ss^3$ as circle bundle over $\Ss^2$, with projection map $\pi$ (the Hopf map). A function $f$ is called \textit{basic} if it is a first integral of $R$, i.e., $R\cdot \nabla f=0$, which means that $f=\ov{f} \circ \pi$ for some function $\ov{f}$ on $\Ss^2$. As $\pi: \Ss^3 \to \Ss^2\big(\frac{1}{2}\big)$ is a Riemannian submersion, we have that $\Delta_{\Ss^3}f = (\Delta_{\Ss^2}\ov{f}) \circ \pi$ so $f$ is an eigenfunction of the Laplacian on $\mathbb S^3$ if and only if $\ov{f}$ is an eigenfunction of the Laplacian on $\mathbb S^2\big(\frac12\big)$ (understood as the sphere of radius $1/2$ in $\mathbb R^3$). We observe that the spectrum of the Laplacian on $\Ss^2\big(\frac{1}{2}\big)$ is given by $\{4k(k+1):k\in \mathbb N\}$, and hence it is of the form $\lambda(\lambda-2)$ ($\lambda>0$) if and only if $\lambda=2+2k$, so all the eigenfields we produce with this method have even eigenvalue. Notice also that the $\Ss^1$-invariance of the vector field~\eqref{raf} means that $V$ is a basic vector field (i.e. $[R, V]=0$), so it projects down to a vector field on $\Ss^2$.

If, in addition, the nodal set of the eigenfunction $\ov{f}:\Ss^2\big(\frac12\big) \to \RR$ is non-empty and regular, which is the case for a generic eigenfunction, cf. Proposition~\ref{density} (the fact that the nodal set is non-empty is immediate because any eigenfunction with positive eigenvalue has zero mean), then the vector field $V$ in~\eqref{raf} is non-vanishing and, since the projection of the characteristic surface $\pi(\Gamma_R)$ (see the proof of Proposition~\ref{P:OT} for its definition) coincides with the nodal set of $\ov{f}$, the induced contact structure is overtwisted by Giroux's characterization of tight $\Ss^1$-invariant contact structures~\cite[Proposition 4.1b]{gir} in the particular case of $\mathbb S^3$.

Summarizing, this construction gives for each even eigenvalue $\lambda=2m$ of the curl on $\mathbb S^3$ a plethora of nonvanishing eigenfields that are associated with overtwisted contact structure. However, if we want to compute the corresponding homotopy classes, i.e., the Hopf invariant of the associated maps $\varphi_m$ as introduced before Lemma~\ref{L:comp}, we have to use concrete eigenfunctions $\ov{f}$ of the Laplacian on $\mathbb S^2$, as done in Section~\ref{S:sphere}; in general, this is a difficult task.

\subsection{Other examples in $\mathbb S^3$ and $\mathbb T^3$}

The nonvanishing curl eigenfields constructed in Section~\ref{S:sphere} (and Subsection~\ref{SS:Raf}) for every even eigenvalue obviously do not exhaust all nonvanishing eigenfields on $\mathbb S^3$. The following shows an example, but it turns out to be contact homotopic to $V_2$ (the curl eigenfield constructed in Section~\ref{S:sphere} with eigenvalue~4). It would be interesting to know whether there exist nonvanishing curl eigenfields on $\mathbb S^3$ associated with overtwisted contact structures with Hopf invariant different from $0$ or $-1$.

\begin{ex}\label{nonKKPS2}
In terms of the global orthonormal frame $\{R,X_1,X_2\}$ introduced in Section~\ref{S:sphere}, the vector field
$$V=\tfrac{1}{2} (x_1^2 + 4 x_1 x_2 - x_2^2 + 2 (y_1^2 - y_2^2))R-(x_2 y_1 + x_1 y_2)X_1-(x_1^2 - x_2^2 + y_1 y_2)X_2\,,$$
(understood as the restriction to $\mathbb S^3$) is a nonvanishing curl eigenfield with eigenvalue $4$, and is neither of the form~\eqref{kkps} nor \eqref{raf}. By direct inspection we find out that $V$ and $V_2$ have only points where they are negatively aligned (i.e. $C_{+}=\emptyset$). Then, by Proposition~\ref{gencont}, we conclude that the contact structure induced by $V$ is contact homotopic to the one induced by $-V_2$, which is overtwisted and of Hopf invariant $-1$.
\end{ex}

Using Proposition~\ref{gencont}, we can also construct nonvanishing curl eigenfields on $\mathbb T^3$ different from those presented in~\eqref{genericbelt3}, but still associated with tight contact structures. It is interesting to compare this example with the following \textit{tightness criterion} that can be deduced from~\cite[Theorem 1.5]{ekm} along the same lines as~\cite[Theorem 1.9]{ekm}: on the flat 3-torus, a non-vanishing curl eigenfield $V$ is associated with a tight contact structure if $d_g:=\max_{\Tt^3}\abs{\nabla \ln \abs{V}} \in [0, \pi^{-1}]$.

\begin{ex}\label{torus_ex}
$V=\sin mx_3 \dif x_1 +\cos mx_3 \dif x_2$ and $W=-\sin mx_2 \dif x_1 +\cos mx_2 \dif x_3$ are nonvanishing Beltrami fields for the eigenvalue $m \in \NN$ on $\mathbb T^3$. By direct inspection we find out that $V$ and $W$ have collinearity points and, since they have unit norm, the constant in~\eqref{eq.c0} is $c_0=1$. By Proposition \ref{gencont}, the contact structure induced by the Beltrami field $V+cW$, for any $\abs{c}<1$, is contact homotopic to the one induced by $V$ which is the $m^{th}$ standard tight contact structure in \eqref{standardT3}; compare with \cite[Proposition 4.8]{etn}. Notice that $V+cW$ is a (non-vanishing if $\abs{c}<1$) curl eigenfield of the form \eqref{ansatzT3} after an obvious permutation of coordinates; in particular, the tightness of the induced contact structure can be equally proved using Giroux characterization \cite[Proposition 4.1b]{gir}. We also remark that $V+cW$ (and its dual 1-form $\alpha$) has  non-constant norm and that $d_g:=\max_{\Tt^3}\abs{\nabla \ln \abs{\alpha}} =\frac{m c}{1- c^2}$ may lie outside the aforementioned ``tightness interval'' $[0, \pi^{-1}]$, for each $m \in \NN^*$ and $c \in(c_1,1)$ with $c_1$ appropriately chosen, depending on $m$.
\end{ex}

\subsection{An application to planar open book decompositions}

Recall~\cite{etnbook} that any overtwisted contact structure on a closed 3-manifold is supported by a planar open book decomposition. As an explicit realization of this fact, we have the following result for $V_2$ and $V_3$. An analogous computation of the corresponding planar open book decomposition is feasible (but tedious) for all $V_m$, $m\geq 2$. We observe that $\widetilde \pi$ of the second planar open book below is not a Milnor fibration.

\begin{pr}
The contact structure associated with the nonvanishing curl eigenfield $V=-\frac{3}{2}V_2$ is supported by the standard open book
$$\pi_{-}: \Ss^3 \setminus B \to \Ss^1\,, \qquad \pi_{-}(z_1, z_2)= \frac{z_1 \ov{z}_2}{\abs{z_1 \ov{z}_2}}$$
with (oriented) binding $B=H^{-}=\{ (z_1, z_2)\in \Ss^3 : z_1 \ov{z}_2 =0\}=\{s=0;\phi_1=t\}\cup \{s=\pi/2;\phi_2=-t\}$ (negative Hopf link).

The contact structure associated with the nonvanishing curl eigenfield $\widetilde V=2V_3$ is supported by the open book
\begin{equation*}
\widetilde \pi:\Ss^3\setminus B \to \Ss^1\,, \qquad
\widetilde \pi(z_1, z_2)= \frac{z_1 z_2(\ov{z}_1 - \ov{z}_2)^2}{\abs{z_1 z_2}\abs{z_1 - z_2}^2}
\end{equation*}
with (oriented) binding $\widetilde B=\{ (z_1, z_2)\in \Ss^3 : z_1 z_2(\ov{z}_1 - \ov{z}_2) =0\}=\{s=0;\phi_1=t\}\cup \{s=\pi/4; \phi_1=\phi_2=-t\}\cup \{s=\pi/2;\phi_2=t\}$, $t\in[0,2\pi)$.
\end{pr}

\begin{proof}
Using the explicit formula of $V_2$ presented in Example~\ref{concretexi}, it is easy to check that the vector field $V$ is (positively) tangent to $B$ and transverse to the pages of $\pi_{-}$, thus proving the fact that the contact structure given by the dual 1-form $\alpha = V^\flat$ is supported by (or compatible with) the planar open book $(\pi_-, H^-)$. Indeed, $\alpha$ has the form:
$$
\alpha=(3\cos^2 s-2)\cos^2 s\,\dif \phi_1+(3\cos^2 s-1)\sin^2 s\,\dif \phi_2\,,
$$
and for each $\phi\in \Ss^1$, the $\phi$-page of the open book $\pi_-$ can be parameterized using the coordinates $(s,\phi_1)$ as
$$
\{(\cos s\, e^{\ii \phi_1},\sin s\, e^{\ii (\phi_1-\phi)}):(s,\phi_1)\in(0,\pi/2)\times [0,2\pi)\}\,.
$$
A straightforward computation shows that $\dif \alpha$ restricted to the $\phi$-page is
$$
d\alpha|_{\phi-\text{page}}=-2\sin 2s\,\dif s\wedge \dif \phi_1\,,
$$
so taking into account that $\{-\partial_s,\partial_{\phi_1}-\partial_{\phi_2}\}$ is a positively oriented global frame on each page of the open book (notice that the outward vector field on the sphere is $-\partial_s$, and $\partial_{\phi_1}-\partial_{\phi_2}$ is a vector field positively aligned with the negative Hopf link on $B$), we conclude that
$$
d\alpha|_{\phi-\text{page}}(-\partial_s,\partial_{\phi_1}-\partial_{\phi_2})=2\sin 2s >0\,.
$$
Therefore, $\dif \alpha$ restricts to an area form on the pages of the open book.

Finally, we observe that $\alpha=\dif \phi_1$ on $\{s=0\}$ and $\alpha=-\dif\phi_2$ on $s=\pi/2$, so
$$
\alpha|_B(\partial_{\phi_1}-\partial_{\phi_2})>0\,,
$$
thus proving that the contact structure is positively transverse to $B$. These two conditions are precisely the definition of the compatibility between a contact structure and an open book, and the first part of the proposition follows.

For the case of $\widetilde V$, we can also use the explicit formula of $V_3$ in Example~\ref{concretexi} to show that $\widetilde V$ is (positively) tangent to $\widetilde B$ and (positively) transverse to the pages of $\widetilde \pi$, thus proving that the contact structure associated to $\widetilde V$ is supported by (or compatible with) the planar open book $(\widetilde \pi, \widetilde B)$. Indeed, $\widetilde V \big \vert _{s=0}=\partial_{\phi_1}$, $\widetilde V\big \vert _{s=\pi/2}=\partial_{\phi_2}$ and $\widetilde V \big \vert _{s=\pi/4}=-\frac{1}{2}\partial_{\phi_1} - \frac{1}{2}\partial_{\phi_2}$ so $\widetilde V$ is (positively) tangent to $\widetilde B$. Moreover, noticing that the open book in coordinates $(s,\phi_1,\phi_2)$ takes the form
$$
\widetilde \pi(\cos s e^{\ii \phi_1}, \sin s e^{\ii \phi_2})= e^{\ii \Theta(s,\phi_1,\phi_2)}
$$
with
$$
\Theta(s,\phi_1,\phi_2):=\phi_1 + \phi_2 - 2\arctan\frac{\cos s \sin \phi_1 - \sin s \sin \phi_2}{\cos s \cos \phi_1 - \sin s \cos \phi_2}\,,
$$
and the integral curves of $\widetilde V$ for a point $s\in(0,\pi/4)\cup (\pi/4,\pi/2)$, $\phi_1=\phi_2=0$, are given by
$$
\phi_1(t)=\Big(\frac32 -6\cos^2 s+5\cos^4 s\Big)t\,,\qquad \phi_2(t)=\Big(\frac12 -4\cos^2 s+5\cos^4 s\Big)t\,,
$$
a straightforward computation shows that
$$
\frac{d\Theta(s,\phi_1(t),\phi_2(t))}{dt}=\frac{\cos^2 2s}{1-\cos(t\cos 2s)\sin 2s}>0\,,
$$
thus implying that $\widetilde V$ is positively transverse to the pages of the open book, as we wanted to prove.
\end{proof}

\subsection{Nonvanishing Beltrami fields with nonconstant proportionality factor} Throughout this work we have considered nonvanishing curl eigenfields. In the following, we provide examples of Beltrami fields for which the proportionality factor is no longer constant. We shall see that they induce tight contact structures.

\begin{ex}
The following family (depending on two parameters $k$ and $\ell$) of nonvanishing, globally defined smooth vector fields on $\Ss^3$,
$$
V=\frac{1}{k^2 \sin^2 s+ \ell^2 \cos^2 s}(\ell \partial_{\phi_1} + k \partial_{\phi_2})\,, \qquad k, \ell \in \NN\,, k \ell \neq 0\,,
$$
are Beltrami fields with (nonvanishing, nonconstant) proportionality factor
$$f(\cos s e^{i\phi_1}, \sin s e^{i\phi_2})=\frac{2 k \ell}{k^2 \sin^2 s+ \ell^2 \cos^2 s}\,.$$
Notice that $f=2k\ell\abs{V}^2$. The induced contact structures are tight, as it can be checked using the contact homotopy
$$
\alpha_t=\frac{(1+t(\ell-1))\cos^2 s \ \dif \phi_1 + (1+t(k - 1)) \sin^2 s \ \dif \phi_2}{(1+t(\ell-1))^2 \cos^2 s + (1+t(k - 1))^2 \sin^2 s}.
$$
between the $1$-form $V^\flat$ and the standard contact form $\eta$ on $\mathbb S^3$ (see Section~\ref{S:sphere}).
\end{ex}

\begin{ex}\label{ex_final}
Applying~\cite[Theorem 1.9]{ekm} we can deduce that the following nonvanishing Beltrami field on $\mathbb T^3$ (with nonconstant proportionality factor) induces a tight contact structure
$$
V=\cos(F(x_3) - \tfrac{\pi}{4}) \partial_{x_1} + \sin(F(x_3) - \tfrac{\pi}{4}) \partial_{x_2}\,,
$$
where $F$ is any function such that $F'$ is periodic and $F'(x_3) <0$. It can be easily checked that $\cu V = f V$ with $f(x_3)=-F'(x_3)$. In the particular case $F(x_3)=-2x_3+ \cos x_3$, the induced contact form $\alpha:=V^\flat$ is contact homotopic to a unit length eigenfield of curl with eigenvalue $2$, through the explicit homotopy:
$$\alpha_t= \tfrac{1}{\sqrt{2}}\big(\cos (2x_3 -t\cos x_3)-\sin (2x_3 -t\cos x_3)\big)\dif x_1 - \tfrac{1}{\sqrt{2}}\big(\cos (2x_3 -t\cos x_3)+\sin (2x_3 -t\cos x_3)\big)\dif x_2\,.$$
This homotopy is contact because
$$\alpha_t\wedge \dif \alpha_t=(2+t\sin x_3)\dif x_1\wedge \dif x_2\wedge \dif x_3\,.$$
Therefore $\alpha=\alpha_1$ is contact isotopic to
$$\alpha_0=\tfrac{-1}{\sqrt 2}\eta_2+\tfrac{1}{\sqrt 2}(\cos 2x_3 \dif x_1-\sin 2x_3 \dif x_2)\,,$$
where $\eta_2$ is defined in~\eqref{standardT3}. Since $\alpha_0$ is the sum (up to a constant proportionality factor) of $\eta_2$ and a pointwise perpendicular curl eigenform with eigenvalue $\lambda=2$, an easy application of Proposition~\ref{gencont} implies that it is contact homotopic to $\eta_2$ ($\alpha_0$ and $\eta_2$ do not align at any point). In conclusion, the induced contact form $\alpha$ is in the second tight contactomorphic class.
\end{ex}

\section*{Acknowledgments}
The authors are extremely grateful to John Etnyre, Rafal Komendarczyk and Patrick Massot for their very useful remarks and suggestions on a previous version of the manuscript and to Willi Kepplinger for pointing out to us an imprecision in the former statement of Proposition \ref{gencont}. D.P.-S. was supported by the grants MTM PID2019-106715GB-C21 (MICINN) and Europa Excelencia EUR2019-103821 (MCIU). This work was partially supported by the ICMAT--Severo Ochoa grant CEX2019-000904-S.


\begin{thebibliography}{99}

\bibitem{specfunc} G.E. Andrews, R. Askey, R. Roy. Special functions. Cambridge University Press, 1999.

\bibitem{arn} V.I. Arnold, B. Khesin.
Topological Methods in Hydrodynamics. Springer, New York, 1998.

\bibitem{baird} P. Baird, \textit{A B\"ochner technique for harmonic mappings from a 3-manifold to a surface}. Ann. Global Anal. Geom. 10 (1992) 63--72.

\bibitem{BW}
P. Baird, J.C. Wood. Harmonic morphisms between Riemannian manifolds. Oxford Univ. Press, New York, 2003.


\bibitem{blair} D.E. Blair. Riemannian geometry of contact and symplectic manifolds. Birkhauser, Basel, 2010.


\bibitem{CMPP}
R. Cardona, E. Miranda, D. Peralta-Salas, F. Presas, \textit{Universality of Euler flows and flexibility of Reeb embeddings}. Preprint, 2019.

\bibitem{ch} S.S. Chern, R.S. Hamilton, \textit{On Riemannian metrics adapted to three-dimensional contact manifolds}. With an appendix by Alan Weinstein. Lecture Notes in Math. 1111,  Workshop Bonn 1984,  279--308, Springer, Berlin, 1985.

\bibitem{Const}
P. Constantin, A. Majda, \textit{The Beltrami spectrum for incompressible fluid flows}. Comm. Math. Phys. 115 (1988) 435--456.


\bibitem{delellis}
C. De Lellis, L. Sz\'ekelyhidi, \textit{The Euler equations as a differential inclusion}. Ann. of Math. 170 (2009) 1417-1436.

\bibitem{duf}
E. Dufraine, \textit{About homotopy classes of non-singular vector fields on the three-sphere}. Qual. Theor. Dyn. Sys. 3 (2002) 361--376.

\bibitem{ES} J. Eells, J.H. Sampson,
\textit{Harmonic mappings of Riemannian manifolds}. Amer. J. Math. 86 (1964) 109--160.

\bibitem{eells} J. Eells, P. Yiu, \textit{Polynomial harmonic morphisms between Euclidean spheres}. Proc. Amer. Math. Soc. 123 (1995) 2921--2925.

\bibitem{Eli} Y. Eliashberg, \textit{Classification of overtwisted contact structures on 3-manifolds}. Invent. Math. 98 (1989) 623--637.


\bibitem{eps} A. Enciso, R. Luc\`a, D. Peralta-Salas, \textit{Vortex reconnection in the three dimensional Navier-Stokes equations}. Adv. Math. 309 (2017) 452--486.

\bibitem{TAMS}
A. Enciso, D. Peralta-Salas, \textit{Nondegeneracy of the eigenvalues of the Hodge Laplacian for generic metrics on $3$-manifolds}. Trans. Amer. Math. Soc. 364 (2012) 4207--4224.

\bibitem{Annals}
A. Enciso, D. Peralta-Salas, \textit{Knots and links in steady
  solutions of the Euler equation}. Ann. of Math. 175 (2012) 345--367.

\bibitem{Acta}
A. Enciso, D. Peralta-Salas, \textit{Existence of knotted vortex tubes in
steady Euler flows}.  Acta Math. 214 (2015) 61--134.

\bibitem{ARMA}
A. Enciso, D. Peralta-Salas, \textit{Beltrami fields with a nonconstant proportionality factor are rare}. Arch. Rat. Mech. Anal. 220 (2016) 243--260.

\bibitem{ept0}
A. Enciso, D. Peralta-Salas, F. Torres de Lizaur, \textit{Knotted structures in high-energy Beltrami fields on the torus and the sphere}. Ann. Sci. \'Ec. Norm. Sup. 50 (2017) 995--1016.

\bibitem{ept} A. Enciso, D. Peralta-Salas, F. Torres de Lizaur, \textit{High-energy eigenfunctions of the Laplacian on the torus and the sphere with nodal sets of complicated topology}. Springer Proc. Math. \& Stat. 346 (2021) 245--261.

\bibitem{etnbook} J.B. Etnyre, \textit{Planar open book decompositions and contact structures}. Int. Math. Res. Not. 79 (2004) 4255--4267.

\bibitem{etn} J.B. Etnyre, R. Ghrist, \textit{Contact topology and hydrodynamics I. Beltrami fields and the Seifert conjecture}. Nonlinearity 13 (2000) 441--458.

\bibitem{EG00}
J. Etnyre, R. Ghrist, \textit{Contact topology and hydrodynamics
III. Knotted orbits}. Trans. Amer. Math. Soc. 352 (2000) 5781--5794.

\bibitem{ekm} J.B. Etnyre, R. Komendarczyk, P. Massot,
\textit{Tightness in contact metric 3-manifolds}. Invent. Math. 188 (2012) 621--657.

\bibitem{ge} J. Ge, Y. Huang, 1/4-\textit{pinched contact sphere theorem}. Asian J. Math. 20 (2016) 893--902.

\bibitem{geighand} H. Geiges, \textit{Contact geometry}, in Handbook of differential geometry, vol. II (2006): 315--382.


\bibitem{ghri} R. Ghrist, R. Komendarczyk, \textit{Overtwisted energy-minimizing curl eigenfields}. Nonlinearity 19 (2006) 41--51.



\bibitem{gir} E. Giroux,  \textit{Structures de contact sur les vari\'et\'es fibr\'ees en cercles au dessus d'une surface}. Comment. Math. Helv. 76 (2001) 218--262.


\bibitem{glu} H. Gluck, W. Gu, \textit{Volume-preserving great circle flows on the 3-sphere}. Geom. Dedicata 88 (2001) 259--282.

\bibitem{kkpspaper} B. Khesin, S. Kuksin, D. Peralta-Salas, \textit{KAM theory and the 3D Euler equation}. Adv. Math. 267 (2014) 498--522.

\bibitem{borsuk} M. Knebusch, \textit{An algebraic proof of the Borsuk--Ulam theorem for polynomial mappings}. Proc. Amer. Math. Soc. 84 (1982) 29--32.

\bibitem{komtams}
R. Komendarczyk, \textit{On the contact geometry of nodal sets}. Trans. Amer. Math. Soc. 358 (2006) 2399--2413.

\bibitem{komphd}
R. Komendarczyk, \textit{Nodal sets and contact structures}. ProQuest LLC, Ann Arbor, MI, 2006. Thesis (Ph.D.), Georgia Institute of Technology.

\bibitem{kom} R. Komendarczyk, \textit{Tight Beltrami fields with symmetry}. Geom. Dedicata 134 (2008) 217--238.

\bibitem{lee} J.M. Lee. Introduction to Smooth Manifolds. Springer, New York, 2013.

\bibitem{lisi}
S. Lisi, \textit{Dividing sets as nodal sets of an eigenfunction of the Laplacian}. Algebr. Geom. Topol. 11 (2011) 1435--1443.

\bibitem{Marsh}
G.E. Marsh, Force-free magnetic fields: solutions, topology and applications. World Scientific, Singapore, 1996.

\bibitem{Moff}
H.K. Moffatt, \textit{Magnetostatic equilibria and analogous Euler
  flows of arbitrarily complex topology I}. J. Fluid Mech. 159
  (1985) 359--378.

\bibitem{ps}
D. Peralta-Salas, \textit{Selected topics on the topology of ideal fluid flows}. Int. J. Geom. Meth. Mod. Phys. 13 (2016) 1630012.

\bibitem{pss} D. Peralta-Salas, R. Slobodeanu, \textit{Energy minimizing Beltrami fields on Sasakian $3$-manifolds}. Int. Math. Res. Not. 2021 (2021) 6656--6690.


\bibitem{reed} D. Reed, \textit{Foundational electrodynamics and Beltrami vector fields}, in Advanced Electromagnetism: Foundations, Theory and Applications, pp. 217-249 (1995).

\bibitem{sz} G. Szeg\"o, \textit{On an inequality of P. Tur\'an concerning Legendre polynomials}. Bull. Amer. Math. Soc. 54 (1948) 401--405.

\bibitem{Wi96} G. Wiegmink, \textit{Total bending of vector fields on the sphere $\mathbb S^3$}. Diff. Geom. Appl. 6 (1996) 219--236.

\bibitem{kan} K. Yutaka, \textit{The classification of tight contact structures on the 3-torus}. Comm. Anal. Geom. 5 (1997) 413--438.

\end{thebibliography}
\end{document}